\documentclass[11pt]{amsart}
 \usepackage{graphicx}
\usepackage{xstring}
\usepackage{forloop}
\usepackage{bbm, amsmath, amsfonts, amscd, latexsym, amsthm, amssymb, graphicx, mathrsfs, comment}
\usepackage[colorlinks=true]{hyperref}

\usepackage{abstract}
\usepackage{mathtools}
\usepackage{caption}
\usepackage[dvipsnames]{xcolor}

\usepackage{tikz-cd}
\tikzset{
  symbol/.style={
    draw=none,
    every to/.append style={
      edge node={node [sloped, allow upside down, auto=false]{$#1$}}}
  }
}

\makeatletter
\newif\if@check@engine  \@check@enginetrue 
\makeatother
\usepackage[usenames,dvipsnames]{pstricks}

\usepackage{booktabs}

\newtheorem{theor}{\hspace{1cm}{\sc Theorem}}[section]
\newtheorem{utver}[theor]{\hspace{1cm}{\sc Proposition}}
\newtheorem{sledst}[theor]{\hspace{1cm}{\sc Corollary}}
\newtheorem{lemma}[theor]{\hspace{1cm}{\sc Lemma}}
\newtheorem{conj}[theor]{\hspace{1cm}{\sc Conjecture}}

\newtheorem*{utver*}{\hspace{1cm}{\sc Proposition}}
\theoremstyle{definition}

\newtheorem{defin}[theor]{\hspace{1cm}{\sc Definition}}
\newtheorem{exa}[theor]{\hspace{1cm}{\sc Example}}
\newtheorem{observ}[theor]{\hspace{1cm}{\sc Observation}}
\newtheorem{rem}[theor]{\hspace{1cm}{\sc Remark}}
\newtheorem{prb}[theor]{\hspace{1cm}{\sc Problem}}

\newcommand{\sgn}{\mathop{\rm sgn}\nolimits}

\newcommand{\Vol}{\mathop{\rm Vol}\nolimits}

\newcommand{\id}{\mathop{\rm id}\nolimits}

\newcommand{\conv}{\mathop{\rm conv}\nolimits}

\newcommand{\Trop}{\mathop{\rm Trop}\nolimits}

\newcommand{\MV}{\mathop{\rm MV}\nolimits}

\newcounter{idx}

\newcommand{\rotraise}[1]{
  \StrLen{#1}[\slen]
  \forloop[-1]{idx}{\slen}{\value{idx}>0}{
    \StrChar{#1}{\value{idx}}[\crtLetter]
    \IfSubStr{tlQWERTZUIOPLKJHGFDSAYXCVBNM}{\crtLetter}
      {\raisebox{\depth}{\rotatebox{180}{\crtLetter}}}
      {\raisebox{1ex}{\rotatebox{180}{\crtLetter}}}}
}

\newcommand{\rg}{\mathfrak{g}}
\newcommand{\ru}{\mathfrak{u}}
\newcommand{\rv}{\mathfrak{v}}
\newcommand{\rw}{\mathfrak{w}}

\DeclareMathOperator{\GCD}{GCD}

\renewcommand{\l}[1]{{\color{blue} \sf  #1}}

\renewcommand{\emph}[1]{{\it {\color{NavyBlue} #1}}}

\def\R{\mathbb R}

\def\Z{\mathbb Z}

\def\C{\mathbb C}
\def\CC{({\mathbb C}^\star)}

\def\CP{\mathbb C\mathbb P}

\begin{document}

\begin{center}{\Large \sc Bernstein-Kouchnirenko-Khovanskii with a symmetry}

\vspace{3ex}

{\sc Alexander Esterov and Lionel Lang}
\end{center}

\begin{abstract}
A generic polynomial $f(x,y,z)$ with a prescribed Newton polytope defines a symmetric spatial curve $f(x,y,z)=f(y,x,z)=0$. We study its geometry: the number, degree and genus of its irreducible components, the number and type of singularities, etc.
and discuss to what extent these results generalize to higher dimension and more complicated symmetries. 

As an application, we characterize generic one-parameter families of complex univariate polynomials, whose Galois group differs from the complete symmetric group.
\end{abstract}
{\let\thefootnote\relax

\tableofcontents

\section{Introduction}

{\bf Complete intersections with a symmetry.} We study the interface between two active but distinct subject matters: the study of generic complete intersections in the algebraic torus (known as the theory of Newton polytopes), and a growing collection of special varieties (in various parts of geometry with various motivations), which can be represented as complete intersections with an action of a finite symmetry group.

A {\it generic complete intersection} is given by a system of polynomial equations with monomials $x^a=x_1^{a_1}\cdot\ldots\cdot x_n^{a_n}$ from a prescribed finite set $A\subset\Z^n, a\in A$, and indeterminate coefficients. Geometry of such a special variety $V$ is naturally related to the geometry of its {\it Newton polytope} -- the convex hull of the set of monomials $A\subset\Z^n$.

This relation (known as the theory of Newton polytopes) is thoroughly studied by many authors starting from \cite{k75}, \cite{bernst}, \cite{varch} and \cite{kh75}, with two major motivations. On the one hand, many important special varieties (such as incidence varieties in enumerative  geometry or members of classifications in singularity theory), can be represented as generic complete intersections, and thus studied using Newton polytopes. 

On the other hand, a key method to construct a variety with prescribed geometry is to find a lattice polytope $N$ with respective geometric properties, so that the theory of Newton polytopes imparts these properties to the generic complete intersection with the Newton polytope $N$. Noteworthy examples of such constructions include the patchworking technique to construct real algebraic varieties with prescriped topology, leading e.g. to important advances in Hilbert's 16'th problem \cite{viropatch}, \cite{sturmfpatch}, \cite{itenbpatch}, and the reflexive polytope technique to construct Fano and Calabi--Yau varieties with prescribed geometry, leading to important advances in e.g. birational geometry and combinatorial mirror symmetry \cite{bb94}, \cite{ckp}.

\smallskip

All of the cited topics apply the theory of Newton polytopes and, at the same time, see a growing interest in introducing symmetry groups into their settings, starting from the simplest group $\Z/2\Z$ (one involution). This leads to complete intersections (or just hypersurfaces) that are generic modulo this symmetry: see e.g. \cite{os} and \cite{bru} in real algebraic geometry and e.g. \cite{asing}, \cite{ssing}, \cite{gr19} in singularity theory. An important way to construct examples and describe members of classifications of nearly rational and Calabi--Yau varieties are double covers over generic weighted or toric complete intersections (see e.g. \cite{pukhl} and \cite{yau}): they can be themselves represented as toric complete intersections, which are generic, but only modulo a $\Z_2$ symmetry. Weighted complete intersections with further symmetries, important for 
these topics, appear e.g. in \cite{bayley}, \cite{fnp} (several involutions) and 
\cite[Section 6]{prokh},  \cite[Section 5]{ejm} (higher symmetric groups). 

There are obviously many more sources of such complete intersections that are generic modulo symmetry, in different parts of geometry. Some more are mentioned at the end of the paper, but we do not aim to give an exhaustive survey. The main point for us is that such symmetric complete intersections are rarely generic enough to suit the classical Newton polytope theory: this motivates to adapt the theory itself to the presence of symmetry.

\smallskip

Such an adaptation looks ambitious for systems of arbitrarily many equations in arbitrarily many variables (and all the more so for arbitrary symmetry groups), so we start with a model example: two equations $f(x,y,z)=g(x,y,z)=0$, whose coefficients are indeterminate modulo the simplest $\Z_2$ symmetry $f(x,y,z)=g(y,x,z)$. Its careful study alone already gives nontrivial applications and even ideas on the possible structure of the answer in the general case.

Namely, the geometry of the spatial curve $f(x,y,z)=f(y,x,z)=0$ for a generic polynomial $f$ with the Newton polytope $N$, turns out to uniformly depend on $N$ except for some families of exceptional ``thin'' polytopes. These families can be explicitly classified and strongly resemble excluded minors in combinatorics (though we were unable to formally reduce the question to any known version of the Robertson--Seymour theorem or Rota conjecture). 

In contrast to the pure combinatorics, where sets of excluded minors are notoriously large and nonnconstructive, our model example leads to an optimistic conjecture (Remark \ref{bigq00}) on how to constructively describe ``excluded Newton polytopes'' for symmetric complete intersections of arbitrary dimension and codimension.
(On the combinatorial geometry side, this may give a new useful example of how the geometric nature of the objects studied simplifies the situation with excluded minors.)

\smallskip

Besides illustrating this conjecture, our spatial curve example illustrates one general application: it is unrelated to the aforementioned ones, because it is irrelevant to the classical Newton polytope theory without symmetries. It comes from Galois theory. 

There is a very general and beautiful trick, dating back at least to Camille Jordan, on how to prove that a Galois group $G$ is full symmetric. If $G$ contains a transposition, then it is enough to prove that it is doubly transitive. The later property, for the Galois group of a covering $p:U\to V$, means that the fiber square of $p$ is just transitive, i.e. its total space $W$ is irreducible.

For Galois groups in enumerative geometry, this trick was popularized in \cite{harris}, and later led to important advances in many classical enumerative problems like \cite{sw} and \cite{hk}, sometimes leading to a complete solution \cite{compos}. In this setting, the total space of the covering $U\to V$ is the incidence variety of the enumerative problem, lying in the Cartesian product $S\times V\to V$. As mentioned above, it can be frequently described as a complete intersection $f_\bullet(s,v)=0$ in $S\times V$. Then its fiber product is the symmetric complete intersection $f_\bullet(s,v)=f_\bullet(s',v)=0$ in $S\times S\times V$ (and the same for higher $k$-incidence varieties in $S^k\times V$, having similar importance in the study of Galois groups).

Such complete intersections for many enumerative problems could be targeted by the anticipated ``Newton polytope theory with symmetries''. We illustrate this by applying our spatial curve example to study Galois groups of univariate fewnomials depending on a parameter.

Polynomials of this kind were studied by many authors from the perspective of Galois theory, especially in finite characteristics. A remarkable project in this direction initiated by Abhyankar in \cite{abh1} lead to many beautiful results such as \cite{abh2}. It can borrow the name of one of its papers: ``Nice equations for nice groups'', 
and represents some interesting simple groups as Galois groups of trinomials, whose coefficients are monomials of the parameter.
We study the same problem over the complex numbers, but in full generality --- for polynomials composed of an arbitrary finite collection of monomials.

\vspace{1ex}

{\bf Symmetric spatial curves.}
Consider the complex torus $T\simeq\CC^3$ and its character lattice $M\simeq\Z^3$. A finite set $A\subset M$ gives rise to the space $\C^A$ of (Laurent) polynomials supported at $A$: $$f(u,v,w)=\sum_{a=(a_1,a_2,a_3)\in A} c_au^{a_1}v^{a_2}w^{a_3}$$  (or $f(x)=\sum_{a\in A}c_ax^a,\, x\in T$, in a coordinate-free form).

For two given support sets $A$ and $B\subset M$, the classical results by Kouchnirenko, Bernstein, Khovanskii and others describe the geometry of the smooth spatial curve $f=g=0$ given by generic equations $(f,g)\in\C^A\times\C^B$: particularly the number, degree and genus of the components of this curve: see \cite{bernst}, \cite{kh75} and \cite{kh15}. 

Let us now impose the symmetry with respect to the involution $I:T\to T$ given by the formula $(u,v,w)\mapsto (v,u,w)$ in a coordinate system $(u,v,w):T\xrightarrow{\sim}\CC^3$, and the corresponding involution $I:M\to M$ (denoted by the same letter despite a small abuse of notation).

From now on, we assume that $B:=I(A)$ and $g(x):=f(I(x))$. In this case, the aforementioned classical results almost never apply to the symmetric curve $f=g=0$ even for generic $f\in \C^A$: to start with, the curve is not smooth in general.

The big picture is as follows (see Figure \ref{fig:intro}): for most of finite sets $A\subset M$, the curve $f=g=0$ consists of several disjoint ``simple'' smooth components (such as the one given by the equations $f(x)=0$ and $x=I(x)$) that we call diagonal, and one more ``complicated'' smooth component, intersecting the diagonal ones without tangency and called the proper component. The diagonal components are by definition {\it simple} in the sense that they are isomorphic to a generic plane curve with a certain prescribed Newton polygon; their geometry is thus clear from the classical theory of Newton polytopes. The proper component is not of this form, and its study (most notably, proving its irreducibility) constitutes a major part of our paper.

The exceptional sets $A\subset M$, for which the above picture fails, can be classified explicitly and studied separately. In particular, all of their components turn out to be ``simple'' in the aforementioned sense.

\begin{rem}\label{bigq00}
The latter observation may be of crucial importance for further research on symmetric complete intersections. A priori, we have two seemingly unrelated sets:

a) $\{$support sets $A$ for which the symmetric curve has more than one non-diagonal component$\}$; this set is relatively difficult to describe (to a large extent, our paper is devoted to this description).

b) $\{$support sets $A$ for which every component of the symmetric curve is contained in a proper subtorus coset, i.e. is simple in the aforementioned sense$\}$; this set is relatively easy to describe, see Proposition \ref{pflat}.

Unexpectedly, upon describing them both, we observe {\it a posteriori} that (a) is contained in (b), and even almost coincide with it. If this inclusion could be proved {\it a priori}, it would greatly simplify our work. If this inclusion extends to higher dimensions and codimensions or larger symmetry groups (see Conjecture \ref{conjflat}), this would allow to extend our results to those more general settings.
\end{rem}

\begin{figure}
    \centering
    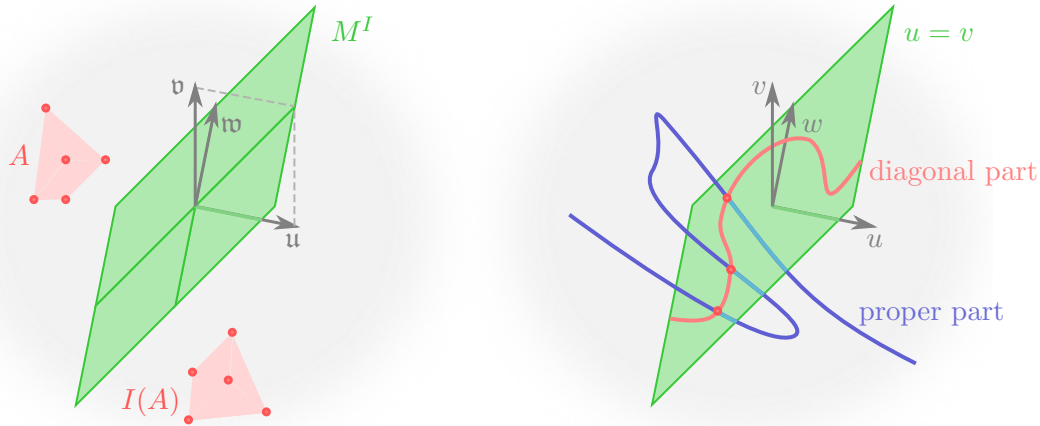
    \caption{On the left, the lattice of monomials with the support $A$ and its symmetric $I(A)$ with respect to the invariant plane $M^I$. On the right, the corresponding symmetric curve $f=f\circ I=0$ in the complex torus, with its proper and diagonal parts}
    \label{fig:intro}
\end{figure}

\vspace{1ex} 

{\bf The structure of the paper.} The background for our story is the theory of Newton polytopes, i.e. the classical facts about the subvarieties defined in algebraic tori by systems of generic equations with prescribed support sets (colloquially known as the  Kouchnirenko--Bernstein--Khovanskii toolkit).

Introducing the symmetry to this classical story in Section \ref{sre}, we first classify the exceptional symmetric curves and describe their geometry case by case, and secondly, describe uniformly the singularities and irreducible components of all the other (non-exceptional) ones. 
The facts regarding singularities and components are then proved in Sections \ref{sconn} and \ref{sirredproof} respectively. 
They are valid under some explicit genericity assumption on $f$, called $I$-nondegeneracy (Definition \ref{defInondeg}).
In the proofs, we make use of a formula for the number of singularities and the Euler characteristics of a symmetric complete intersection (Theorem \ref{thsingnum}), which was obtained in \cite[Theorem 5.6]{crit}, for arbitrary dimension of the ambient algebraic torus.

Finally, in Section \ref{snext}, we discuss the numerous possible generalizations of these facts to the ambient torus of higher dimension, symmetric complete intersections of higher codimension, more complicated symmetries (leading e.g. to questions on geometry of zero loci of Schur polynomials), and, importantly, other ground fields beyond $\C$.

Prior to all of this, Section \ref{sgalois} is devoted to an application of symmetric curves to Galois theory: we characterize generic one-parameter families of complex univariate polynomials, whose Galois group is a complete symmetric group.

\vspace{1ex} 

{\bf An application: Galois groups of fewnomials depending on a parameter.}
Polynomials of this kind were studied by many authors from the perspective of Galois theory, especially in finite characteristics. A remarkable project in this direction initiated by Abhyankar in \cite{abh1} lead to many beautiful results such as \cite{abh2}. It can borrow the name from one of its papers: ``Nice equations for nice groups'' 
and deals with Galois groups of trinomials, whose coefficients are monomials of the parameter.

We study the same problem over the complex numbers, but in full generality --- for polynomials with arbitrary supports.
More specifically, we study how one can permute the non-zero roots of a univariate Laurent polynomial $\varphi(x)=\sum_{j=n}^N c_jx^j$, whose coefficients $c_j=c_j(t)$ are generic Laurent polynomials of a parameter $t$ with prescribed support sets $A_j\subset\Z^1$, $j=n,\ldots, N$. As we vary $t\in\C$ along all possible loops (avoiding the finitely many bifurcation values of $t$ at which the polynomial $\varphi$ has less than $N-n$ roots), the corresponding permutations of the $N-n$ non-zero roots of $\varphi$ form a subgroup $G\subset S_{N-n}$, referred to as the {\it Galois group} of $\varphi$.

Alternatively, given a generic curve $C$ in the $(x,t)$-plane, whose equation $\varphi(x,t)=0$ is supported at the prescribed set $A=\bigcup_{j=n}^N\{j\}\times A_j\subset\Z^2$, the projection of $C$ to the $t$-axis is a ramified covering, and we study its monodromy group $G$.

In order to express $G$ in terms of $A$ (or the $A_j$'s), let $d$ be the maximal integer such that $\varphi$ can be represented as $x^m\psi(x^d)$, with $\psi$ a Laurent polynomial. Combinatorially, $d$ is the GCD of the pairwise differences of all $j$'s such that $c_j(t)\not\equiv 0$ (i.e. $A_j\ne\varnothing$).

\begin{rem}\label{remnonfull}
Since the non-zero roots of $\varphi(x)=x^m\psi(x^d)$ split into $\frac{N-n}{d}$ necklaces of length $d$ (invariant under multiplication by $\sqrt[d]{1}$), at most we can expect to permute these necklaces. 

Thus, whenever $N-n>d>1$ (more than one non-trivial necklace) or $N-n=d>2$ (the length of the necklace is greater than 2), the Galois group is not full symmetric.

Similarly, if $d=1$, and $A$ is contained in an affine line, then $\varphi(x,t)$ up to a monomial multiplier can be represented as $\psi(xt^a)$ for a polynomial $\psi$ and integer $a$. In this case, the covering is trivial, and so is the monodromy.\end{rem}

These obvious obstructions for the Galois group to be full symmetric are the only ones.

\begin{theor}\label{thm:galois}
The Galois group $G$ of a generic polynomial $\varphi\in\C^A$ equals $S_{N-n}$ if and only if 

-- $d=1$, and $A$ is not in an affine line, or 

-- $d=N-n=2$, and $A$ consists of two points whose barycenter is not in the lattice.
\end{theor} 

See the next section for the proof and the exact assumption of general position on $\varphi$ (Definition \ref{defgaloisgen} and Remark \ref{remnondeggalois}). Moreover, in \cite{el} we compute the Galois group for an arbitrary support set $A$. However, that paper relies more heavily on topological methods (such as braids), making it difficult to generalize to arbitrary ground fields and to several variables. For the techniques of the present paper, this looks more realistic: see the very end of the paper for a sketch of extension to fields different from $\C$, and Section 5 of \cite{crit} for what is already known for several variables. For now, applying Theorem \ref{thm:galois} to the rational function $\varphi(x,t)=q_1(x)+q_2(x)t$, we get the following.
\begin{sledst}
Let $\psi(x)=q_1(x)/q_2(x)$ be a rational function with generic numerator and denominator $q_i\in\C^{B_i}$, $B_i\subset\Z^1$. Then the monodromy group of the ramified covering $\psi:\CP^1\to\CP^1$ equals the full symmetric group if and only if the set $B_1\cup B_2$ can not be shifted to a proper sublattice or can be shifted to $\{0,2\}$.
\end{sledst}

We start with proving Theorem \ref{thm:galois}, illustrating how symmetric curves in $\C^3$ come into play.

\section{Generic families of polynomials whose Galois group is full symmetric}\label{sgalois}

{\it Proof of Theorem \ref{thm:galois}.} 
If $A$ does not satisfy the conditions in the theorem, then the Galois group is not full symmetric by Remark \ref{remnonfull}. If $A$ satisfies the second condition, then the Galois group clearly equals $S_2$. So now, it remains to treat the case $d=1$, $N-n>2$, and $\varphi$ is generic.

We study the ramified covering given by the projection $\pi$ of the curve $\varphi(x,t)=0$ in the $(x,t)$-plane to the $t$-axis.
The fiber product of $\pi$ with itself is the projection $\pi^2$ of the curve $$\varphi(x_1,t)=\varphi(x_2,t)=0\eqno{(*)}$$ in the $(x_1,x_2,t)$-space to the $t$-axis.

Then, by Definition \ref{def:diag&properparts}, Observation \ref{ex:diagcomponents&singularities} and Theorem \ref{thintroirr} below, the symmetric curve $(*)$ consists of a single diagonal component (given by the equations $x_1=x_2$ and $\varphi(x_1)=0$) and one more irreducible component for the proper part.

Thus, for any generic $t$, every two points $(x_1,x_2,t),\, x_1\ne x_2$, and $(x'_1,x'_2,t),\, x'_1\ne x'_2$, of the proper part of $(*)$ can be connected with a path $\gamma$ that avoids the singular locus of $(*)$ and whose image $\tilde \gamma$ under $\pi^2$ avoids the branching points of $\pi^2$.
In particular, the monodromy of $\pi^2$ along the loop $\tilde\gamma$ sends the point $(x_1,x_2,t)$ of the fiber of $\pi^2$ over $t$ to the point $(x'_1,x'_2,t)$ of the same fiber.

In terms of the initial covering $\pi$, this means that the monodromy of $\pi$ along $\tilde\gamma$ sends $x_1$ to $x'_1$ and $x_2$ to $x'_2$. We have proved that any pair of distinct points $x_1\ne x_2$ in the fiber of $\pi$ over $t$ can be sent to any other pair of distinct points $x'_1\ne x'_2$ in the same fiber by a suitable element of the monodromy group of $\pi$ (in this situation the monodromy group is said to be 2-transitive).

In particular, the monodromy group contains the transposition $(x_1,x_2)$ if and only if it contains any other transposition $(x'_1,x'_2)$, i.e. is full symmetric. Thus it is enough to prove the following lemma. \hfill$\square$
\begin{lemma}\label{ltranspos}
For a generic polynomial $\varphi$ supported at the set $A\subset\Z^2$ not contained in an affine line and with $d=1$, the ramified covering $\pi$ has at least one branching point whose local monodromy is a transposition.
\end{lemma}

Before we prove the lemma, let us discuss exact conditions of general position on the polynomial $\varphi$ (Definition \ref{defgaloisgen}), which assure that the Galois group is full symmetric.

Let 
$n,n+q,N-Q$ and $N$ be the two smallest and the two largest numbers in the set $B:=\{j\,|\,A_j\ne\varnothing\}$.

Assuming that the closure $\bar C\subset\CP^1\times\C^*$ is smooth, its projection to the $t$-axis $\C^*$ has 
the following critical points:

(i) a point $(0,t)$ for $c_n(t)=0$, of order at least $q$, if $q>1$;

(ii) a point $(\infty,t)$ for $c_N(t)=0$, of order at least $Q$, if $Q>1$;

(iii) a point $(x,t)$ for $\varphi(x,t)=\partial f/\partial x(x,t)=0,\,x\in\C^*$, of order at least 2.

(Note that the order of a critical point $(x_0,t_0)$ for us is the multiplicity of the root $(x_0,t_0)$ of the function $t-t_0$ restricted to the curve $\bar C$; in particular, order 1 means no critical point.)

\begin{defin}\label{defgaloisgen} A polynomial $\varphi\in\C^A$ is said to be $t$-nondegenerate, if:

1) The closure of $C$ in the $A$-toric variety is smooth and transversal to the orbits;

2) all of the listed critical points of the projection $\bar C\to\C^*$ have minimal possible orders;

3) the critical values at any two of these points are distinct, with the necessary exception that every critical point $(x,t)$ of type (iii) has $d-1$ more critical points of the form $(\sqrt[d]{1}\cdot x,t)$ have the same critical value $t$.
\end{defin}

\begin{lemma}\label{galoisgen} 
1) $t$-nondegenerate polynomials $\varphi$ are generic, i.e. form a non-empty Zariski open subset in $\C^A$.

2) The ramified coverings $\bar C\to\C^*$ have the same monodromy group for all $t$-nondegenerate polynomials $\varphi$. 
\end{lemma}

{\it Proof.} 
Genericity of condition (1) in the definition of nondegeneracy is well known. To verify genericity of conditions (2, 3), we can assume wlog that $d=1$ (letting $x^d$ to be the new variable $\tilde x$).
Under this assumption, the space $\C^B$ of all univariate polynomials supported at $B:=\{j\,|\,A_j\ne\varnothing\}$ contains: 

-- the discriminant hypersurface $D$, defined as the closure of the set of all polynomials having a multiple root in $\C^*$; 

-- the coordinate hyperplane $H_j$ defined as the set of all polynomials with zero coefficient at the monomial $x^j$.

The polynomial $\varphi$ defines the map $\Phi:\C^*\to\C^B$ sending $t$ to $\varphi(\bullet,t)\in\C^B$. For generic $\varphi$, this map is transversal to $D\cup H_n\cup H_{n+q}\cup H_{N-Q}\cup H_N$ (see Theorem 
6.6 and Corollary 6.7 in \cite{esv}). Transversality ensures that $\varphi$ is $t$-nondegenerate in the sense of the preceding definition:

-- transversality to $D$ implies that the curve $C$ is smooth: indeed, $D$ is the discriminant of the projection $p$ of the smooth set $S=\{(f,x)\,|\,f(x)=0\}\subset\C^B\times\C^*$ to $\C^B$, so the transversality of $\Phi$ to $D$ implies its transversality to the map $\pi$, thus the fiber product $p\times_{\C^B}\Phi\simeq C$ is smooth;

-- transversality to $D$ also implies that $\Phi$ avoids the singular locus of $D$. Therefore, the type (iii) critical points have order 2 and pairwise distinct critical values, because a polynomial in the smooth part of $D$ has one multiple root, and this root has multiplicity 2 (see e.g. \cite{gkz});

-- transversality to $H_n$ and $H_N$ implies that the closure $\bar C$ is smooth outside $C$;

-- avoiding $H_n\cap H_{n+q}$ and $H_{N-Q}\cap H_Q$ implies that type (i) and (ii) critical points have orders $q$ and $Q$ respectively;

-- avoiding the pairwise intersections of $H_n$, $H_N$ and $D$ implies that critical points of different types cannot have equal critical values.

This proves Part 1 of the lemma. Part 2 is the following evident fact applied to the family of curves $C_\varphi:=($the closure of $\{\varphi(x,t)=0\}$ in the $A$-toric surface$)$ and their coordinate projections $q(x,t):=t$, as $\varphi$ runs over $U:=\{t$-nondegenerate polynomials$\}$. \hfill$\square$
\begin{utver}
Consider proper regular maps of smooth varieties 

$$U\xleftarrow{p} C\xrightarrow{q}\CP^1,$$

such that 

1) the map $p$ is a locally trivial fibration with fiber $C_\varphi:=p^{-1}(\varphi)$ a smooth compact curve;

2) the projection $U\times\CP^1\to U$, restricted to the set $\{(\varphi,t)\,|\,t$ is a ramification point of $q:C_\varphi\to\CP^1\}$, is an (unramified) covering.

Then the topology of the ramified covering $q:C_\varphi\to\CP^1$ (and in particular its monodromy group) is locally constant as a function of $\varphi\in U$.
\end{utver}

{\it Proof of Lemma \ref{ltranspos}.} We shall prove the lemma for an arbitrary $t$-nondegenerate $\varphi$. This is enough, because generic polynomials are $t$-nondegenerate by Lemma \ref{galoisgen}.1.

Denote the number of branching points of types (i), (ii) and (iii) of the ramified covering $\bar C\to\C^*$ by $\sharp_1,\sharp_2$ and $\sharp_3$ respectively. Since the points of types $(i)$ and $(ii)$ correspond to the roots of the polynomials $c_n\in\C^{A_n}$ and $c_N\in\C^{A_N}$ respectively, we have 
$$\sharp_1=\Vol_\Z (A_n) \; \mbox{ and }\; \sharp_2=\Vol_\Z (A_N),$$
where $\Vol_\Z$ is the integer length of the convex hull of a set in $\Z^1$.
By the Riemann-Hurwitz formula, the quantity
$$-(q-1)\cdot \sharp_1-(Q-1)\cdot\sharp_2-\sharp_3$$ equals the the Euler characteristics of the curve $\bar C$, which in turn equals
$$-\Vol_\Z (A)+\Vol_\Z (A_N)+\Vol_\Z (A_n)$$
by the BKK formula applied to the curve $C=\{\varphi=0\}$. Here $\Vol_\Z (A)$ is the lattice area of the convex hull of $A$ (defined as the Euclidean area times $2!$).

As a result, we have $$\sharp_3=\Vol_\Z (A)-Q\Vol_\Z (A_N)-q\Vol_\Z (A_n),$$
which is strictly positive for $d=1$ and $N-n>1$, unless $A$ is contained in an affine line. Thus we have at least one type (iii) ramification point, and the local monodromy at this point is a transposition.
\hfill$\square$

\begin{rem}\label{remnondeggalois}
\textbf{1.} This concludes the proof of Theorem \ref{thm:galois} for generic $\varphi$, without explicit genericity conditions. But actually we can take $\varphi$ to be $t$-nondegenerate as our genericity assumption in Theorem \ref{thm:galois}. Indeed, since all $t$-nondegenerate polynomials $\varphi$ have the same Galois group by Lemma \ref{galoisgen}.2, and almost all of them have the full symmetric Galois group by Theorem \ref{thm:galois}, so do all of them. Thus Theorem \ref{thm:galois} is valid for all $t$-nondegenerate $\varphi$.

\textbf{2.} If we regard $\varphi$ as a polynomial of $x,t$ and one more dummy variable $\tilde x$, then its $I$-nondegeneracy with respect to the involution $I(x,\tilde x,t)=(\tilde x,x,t)$ (Definition \ref{defInondeg}) is the same thing as its $t$-nondegeneracy in the sense of Definition \ref{defgaloisgen}. This clarifies the geometric meaning of the involved Definition \ref{defInondeg} and deepens the relation between the Galois theory of $\varphi$ and the symmetric curve $\varphi(x,t)=\varphi(\tilde x,t)=0$.

\textbf{3.} Of course $\varphi$ does not have to be $t$-nondegenerate in order to have the full symmetric Galois group. Similarly, the symmetric curve does not have to be $I$-nondegenerate in order to have the minimal possible number of irreducible components. It would be interesting to find significantly weaker but still controllable versions of $t$-nondegeneracy and $I$-nondegeneracy. For $t$-nondegeneracy, the problem is addressed in \cite[Section 2.6]{el}.
\end{rem}

\section{Geometry of symmetric curves: rules and exceptions}\label{sre}

If $f$ and $g$ are generic polynomials with prescribed support sets $A$ and $B\subset M\simeq\Z^3$, then the geometry of the curve $f=g=0$ in the torus $T\simeq\CC^3$ is described by the Bernstein-Kouchnirenko-Khovanskii toolkit (see e.g. Section 2 in \cite{crit} for a review).

We now start to study what happens to this theory, if the curve is symmetric with respect to an involution $I:T\to T$ given by the formula $(u,v,w)\mapsto (v,u,w)$ in a coordinate system $(u,v,w)$. More specifically, from now on we study the curve $C\subset T$ given by the equations $f(x)=0$ and $g(x):=f(I(x))=0$ for generic $f$ supported at a given set $A\subset M\simeq\Z^3,\, 1<|A|<\infty$. Recall that the second equation $g$ is supported at the symmetric set $B:=I(A)$, where the symmetry $I:M\to M$ induced by $I:T\to T$ is denoted by the same letter to a small abuse of notation. 

\subsection{Dimension} First of all, the symmetric ``curve'' $f=g=0$ may have the wrong dimension even for a generic polynomial $f\in\C^A$. We start with examples of such sets $A\subset M$.

The fixed point planes $T^I\subset T$  and $M^I\subset M$ of the involution $I$ are defined by the equations $m(x)=1$ for some monomial $m\in M$ 
and $t=0$ for some surjective linear function $t:M\to\Z$. 

We can and will choose $m$ and $t$ so that $t(m)=2$ (rather than $-2$). More specifically, choosing respective coordinate systems $(u,v,w):T\to\CC^3$ and $(\ru,\rv,\rw):M\to\Z^3$ so that $I(u,v,w)=(v,u,w)$ and $I(\ru,\rv,\rw)=(\rv,\ru,\rw)$, we can set $m(u,v,w)=uv^{-1}$ and $t(\ru,\rv,\rw)=\ru-\rv$.

The monomial $m$ generates a 1-dimensional sublattice $M^{-I}$ (which is the fixed point line $\ru+\rv=\rw=0$ of the involution $-I:M\to M$), and $t\in M^*$ generates the one-parametric subgroup $\exp(\C\cdot t)=\{uv=w=1\}\subset T$ that will be denoted by $T^{-I}$.

\begin{defin}
The planes $M^I$ and $T^I$ will be called the {\it diagonal} planes of the lattice and the torus respectively, and the lines $M^{-I}$ and $T^{-I}$ will be called the {\it antidiagonal} lines.
\end{defin}

\begin{exa}\label{exa2dim}
{\bf (Type D)} Let $A$ belong to a plane parallel to $M^I$, then $f/g$ is a monomial, and the set $f=g=0$ is actually the surface $f=0$. The geometry of this surface for generic $f\in\C^A$ is described by the BKK toolkit.

{\bf (Type E)} Let $A$ belong to a line parallel to $M^{-I}$, then $I(A)$ 
is contained in the same line, and the set $f=g=0$ is empty for generic $f\in\C^A$.
\end{exa}
\begin{theor}\label{thdim}
Unless $A$ is of type D or E, 
the set $f=g=0$ is a non-empty curve for generic $f\in\C^A$.
\end{theor}

See Section \ref{ssproofs1} for the proof. In what follows, we never consider such exceptional support sets $A$.

\vspace{1ex}

\subsection{Connected components}
Even if $f=g=0$ is a curve, it may not be connected. To give such examples, recall that a connected subgroup $S$ of the torus $T$ is called a subtorus, and its coset $g\cdot S=\{gs\,|\, s\in S\}$ for any $g\in T$ will be called a shifted subtorus. 

A codimension one subtorus can be given by an equation $\mu(x)=1$ for a primitive monomial $\mu\in M$ (recall that an element $\mu$ of a lattice $M$ is said to be primitive if it is not divisible by an integer number greater than 1, or, equivalently, the quotient group $M/\mu$ is a lattice).

The character lattice of the subtorus $\mu(x)=1$ defined by a primitive monomial $\mu\in M$ is naturally isomorphic to the quotient lattice $M/\mu$. The support of a generic polynomial $f\in\C^A$ restricted to this subtorus equals the image of $A$ under the projection $M\to M/\mu$, which we denote by $A/\mu$.

Dually, a 1-dimensional subtorus can be represented as $\exp(\C\cdot s)$, where $\exp(s_1,s_2,s_3)=(e^{s_1},e^{s_2},e^{s_3})$ is the exponential map from the Lie algebra of $T$, and $s=(s_1,s_2,s_3)$ is an element of this Lie algebra (or, equivalently, a linear function $s_1\ru+s_2\rv+s_3\rw$ on $M$).

\begin{exa}\label{exadisconn} (See Figure \ref{fig:exceptional} below.) Assume that $A$ is not of type D or E.

{\bf (Type C1)} Let $A$ belong to a plane parallel to $M^{-I}$, then $B=I(A)$ is contained in the same plane. The containing plane can be defined by the equation $s=s_0$ for a surjective linear function $s:M\to\Z$ and $s_0\in\Z$. Then the curve $f=g=0$ is the (disjoint) union of several shifted copies of the one-dimensional subtorus $\exp(\C\cdot s)$. The number of these copies for generic $f\in\C^A$ equals the lattice mixed area of $A$ and $B$. 

{\bf (Type C2)} Let $A$ belong to the union of the plane $M^I$ and a line parallel to $M^{-I}$, then the Newton polytope of $f-g$ is a segment of lattice length $q>0$ in $M^{-I}$. In coordinates, we have $f-g=\prod_{i=1}^q(u-\alpha_i v)$ up to a monomial multiplier for pairwise different $\alpha_i\in\C$, thus $f-g=0$ is the (disjoint) union of $q$ shifted copies of the subtorus $T^I$. In each of these copies, the equation $f=0$ defines a plane curve supported at $A/m$. The geometry of this generic plane curve with a given support is described by the BKK toolkit. 
\end{exa}
\begin{rem}
In coordinates, if $I(\ru,\rv,\rw)=(\rv,\ru,\rw)$, then
examples of a line parallel to $M^{-I}$ include $\ru+\rv=\rw=0$ and $\ru+\rv=\rw=1$. Note that they cannot be sent to each other by an affine automorphism of the lattice commuting with the involution $I$, but any other integer line parallel to $M^{-I}$ can be sent to one of these two with a parallel translation of the lattice commuting with the involution $I$. Indeed, the intersection of such a line with $M^I$ has the form either $\mu$ or $\mu+(1/2,1/2,1)$ for some $\mu\in M\simeq\Z^3$, so the parallel translation by $-\mu$ will do.

Similarly, examples of a plane containing the line $M^{-I}$ include $\rw=0$, $\ru+\rv=0$ and $\ru+\rv=1$. They cannot be sent to each other with an affine automorphism of the lattice commuting with the involution $I$, but any other plane containing $M^{-I}$ can be sent to one of these three. 

\end{rem}

\begin{theor}\label{thconn}
Unless $A$ suits one of the preceding examples (D, E, C1 or C2) up to a shift, the set $f=g=0$ is a connected curve for generic $f\in\C^A$.
\end{theor}
See Section \ref{ssproofs1} for the proof.

\begin{rem}
{\bf 1.} Note that the curves from the two parts of Example \ref{exadisconn} are connected as well, if $q$ or the mixed area of $A$ and $B$ respectively equals 1. One can easily classify all sets $A$ satisfying these conditions.

{\bf 2.} It would be interesting to classify all 3-dimensional Newton polytopes of symmetric non-connected curves, i.e. all lattice polytopes in $\R^3$, whose intersection with $\Z^3$ is contained in $M^I\cup M^{-I}$ and does not suit part 1 of this remark.
This problem of lattice geometry resembles (and contains) the problem of classifying empty pyramids, well known e.g. from the theory of multidimensional continued fractions \cite{karp}.
\end{rem}

\subsection{Diagonal components and singularities} Even if the curve $f=g=0$ is connected, it is rarely irreducible, because it usually has {\it diagonal components}.

\begin{observ}\label{ex:diagcomponents&singularities} Assume that $A$ is not of type D, E, or C1.

1) The set $\{f=g=0\}\cap T^I$ is an algebraic curve contained in $\{f=g=0\}$. 
Moreover, for generic $f\in\C^A$, it is an irreducible component of the curve $\{f=g=0\}$.

Indeed, it is a generic plane curve supported at $A/m$. The geometry of such curves is describe by the classical BKK toolkit: in particular, they are irreducible unless $A/m$ is contained in a line, i.e. $A$ is contained in a plane parallel to $M^{-I}$, which is one of the exceptional types of sets $A$ that we have excluded from our consideration.

2) More generally, there are extra irreducible components similar to the above one. To see this, define the integer $d$ in the following equivalent ways:

-- as the index of the lattice generated by $A+I(A)+M^I$ in $M$;

-- as the largest $d$ such that the integers $t(a)$, $a\in A$, are pairwise equal modulo $d$;

-- as the largest $d$ such that $\mu f-(\mu f)\circ I$ is divisible by $u^d-v^d$ for some monomial $\mu$.

Recall that $m\in M$ is the primitive monomial defining the diagonal plane $T^I=\{m=1\}$, and consider the (disjoint) union of $d$ shifted copies of the diagonal plane defined by the equation $m^d=1$. Then the intersection of the curve $\{f=g=0\}$ with each of these shifted copies is a shifted copy of the curve $\{f=g=0\}\cap T^I$. In particular, it is an irreducible component of $\{f=g=0\}$ for generic $f\in\C^A$.
\end{observ}

\begin{defin}\label{def:diag&properparts} Assume that $A$ is not of type D so that $\{f=g=0\}$ is a (possibly empty) curve, see Theorem \ref{thdim}.

1) The number $d=:d(A)$ is called the denominator of the set $A$.

2) The union of the irreducible components $\{f=g=0\}\cap \{m^d=1\} $ of the curve $\{f=g=0\}$ from the preceding Observation will be referred to as the {\it diagonal part} of the symmetric curve.

3) The union of the other irreducible components (i.e. the closure of $\{f=g=0\}\setminus\{m^d=1\}$) will be called the {\it proper part} of the symmetric curve.
\end{defin}

\begin{rem}
Note that, for the support set $A$ of type C1 or C2, every irreducible component of the curve $f=g=0$ is contained in a certain shifted copy of the diagonal subtorus $\{m=c\},\, c\in\C^*$. However, we do not consider all of them diagonal components: for generic $f\in\C^A$, the diagonal ones are exactly those for which $c$ is a root of unity. 

However, this occurs only for types C1 and C2: for any other symmetric curve, non-diagonal irreducible components never belong to a shifted copy of the diagonal subtorus (see Theorem \ref{thintroirr} below).
\end{rem}

For some support sets $A$, the proper part and the diagonal part do not intersect. Besides the types $D$, $E$, $C1$ and $C2$, there is one more such exceptional type.
\begin{exa}{\bf (Type I10)} Let $A$ be not of types $D$, $E$, $C1$ and $C2$, but contained in the union of two lines parallel to each other and to $M^I$. Then the proper part of the curve $\{f=g=0\}$ is empty for generic $f\in\C^A$ (and thus does not intersect the diagonal part). Indeed, up to a monomial multiplier, we can take $f=p+\mu q$ where $p$ and $q$ are univariate polynomials in some monomial $u^av^aw^b$ and $\mu=u^\alpha v^\beta w^\gamma$ is some monomial. Thus the system $f=g=0$ is equivalent to $f=f-g=0$ which in turn is equivalent to $p-\mu q=(\mu-\mu\circ I)\cdot q=0$. If $q=0$, then so is $p$ but the set $\{q=p=0\}$ is generically empty. The other alternative $\mu-\mu\circ I=0$ leads to solutions in $\{m^{\alpha-\beta}=1\}$, that is in the diagonal part.
\end{exa}

\begin{theor}\label{thsing} Consider the symmetric curve $f=g=0$, defined by a generic polynomial $f\in\C^A$, whose support set $A$ is not of type D, E, C1, C2 or I10.

1) The proper part and the diagonal parts of the curve are not empty.

2) The intersection of each of the proper components with each of the diagonal components is non-empty, and has the same cardinality. In particular, the symmetric curve is connected.

3) The proper and diagonal parts of the curve are smooth curves, intersecting transversally (i.e. without tangency) at finitely many points. 
\end{theor}
See Section \ref{ssproofs1} for the proof. The exact number of the intersection points for the proper and the diagonal parts is computed in the subsequent Theorem \ref{thsingnum}.

\subsection{Irreducibility} Even the proper part of the curve $f=g=0$  may be reducible.

\begin{exa}\label{exaI} (See Figure \ref{fig:exceptional}.) Assume that $A$ is not of type D, E, C1 or C2.

{\bf (Type I1)} Let $A$ belong to the union of two planes parallel to $M^I$ (this includes supports sets of type I10). Denoting the intersections of these planes with $A$ by $A_1$ and $A_2$, the polynomial $f$ decomposes into the sum $f_1+f_2,\, f_i\in\C^{A_i}$.

The proper part of the curve $f=g=0$ is given by the equations $f_1=f_2=0$ since, up to a shift, we can write $f_1=\tilde f_1(uv,w)$ and $f_2=\mu\cdot \tilde f_2(uv,w)$ for some monomial $\mu$. Thus the proper part splits into several shifted copies of the one-dimensional subtorus $T^{-I}$. The number of copies equals the lattice mixed area of $A_1$ and $A_2$ (which makes sense, because both $A_1$ and $A_2$ up to a shift belong to the same two-dimensional lattice $M^I$).

{\bf (Type I2)} Let, up to a shift, $A$ belong to the union of the plane $M^I$, some line $L$ parallel to it, and the symmetric line $I(L)$. Denoting the intersection of the plane with $A$ by $A_0$ and the intersection of the two lines with $A$ by $A_\pm$ respectively, the polynomial $f$ decomposes into the sum $f_-+f_++f_0$, with $f_\bullet\in\C^{A_\bullet}$. 

One can easily see that the proper part of the curve $f=g=0$ is given by the equations $f_+(x)-f_-(I(x))=f(x)=0$, supported at $A_+\cup I(A_-)\subset L$ and $A_0$ respectively. 
The convex hull of the first support set is a segment of the form $\{a',a'+m',a'+2m',\ldots,a'+d'm'\}$ for some integer $d'\geq 1$, so, for generic $f\in\C^A$,  the first equation $f_+(x)=f_-(I(x))$ defines $d'$ shifted copies $T_j$ of the two-dimensional subtorus $m'=1$.
In each of these shifted subtori, the second equation $f=0$ defines a generic  plane curve supported at $A/m'$. The geometry of such curves is governed by the BKK toolkit: in particular, they are smooth and irreducible, because $A/m'$ is not contained in a line. The proper part of the curve $f=g=0$ is the (disjoint) union of these curves.
\end{exa}

\begin{figure}[b]
    \centering
    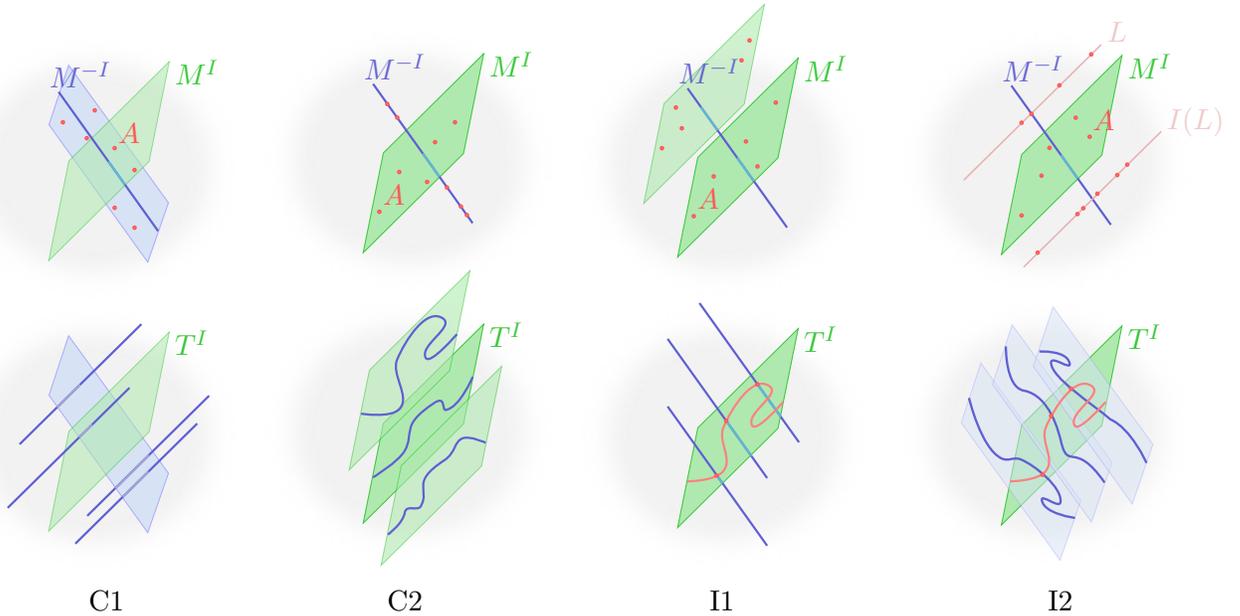
    \caption{The exceptional types C1, C2, I1 and I2. In the first row, we picture the support $A$ in the lattice of monomials. In the second row, we picture the corresponding symmetric curve, with the proper part in blue and the symmetric part in red.}
    \label{fig:exceptional}
\end{figure}

\begin{defin}
If $A$ can be shifted to one of the types D, E, C1, C2, I1 or I2, then it is called {\it exceptional},and so is the symmetric curve $f=g=0$ defined by a generic polynomial $f\in\C^A$ supported at such $A$.
\end{defin}

\begin{theor}\label{thintroirr}
Unless $A$ is exceptional, the proper part of the curve $f=g=0$ for generic $f\in\C^A$ is (1) irreducible, and (2) does not belong to a proper subtorus of $T$. 

\end{theor}
See Section \ref{ssproofs1} for the proof of part (2) and Section \ref{sirredproof} for part (1). An explicit sufficient genericity condition for this theorem is given in Definition \ref{defInondeg} below. 

\begin{rem}
1. In what follows, we never consider  exceptional support sets and curves, unless we explicitly specify the contrary.

2. The notation D, E, C1, C2, I1 and I2 comes from ``the wrong Dimension'', ``Empty'' and ``several Connected or Irreducible components contained in shifted 1 or 2-dimensional subtori'', respectively.

\end{rem}

\section{Enumerative results and genericity conditions for symmetric curves}

{\bf Geometry of the support set.} For what follows, we need some vocabulary associated to the support set $A\subset M\simeq\Z^3$.

Warning: when we say about such finite set $A$ something that makes little sense for finite sets (like having a certain dimension or a certain volume or being parallel to a certain plane), we always imply to say this about the convex hull of $A$.
In particular, a subset $B$ of a finite set $A\subset M\simeq\Z^n$ is called its face if it is the intersection of $A$ with a face of its convex hull. 

We say that the faces $B_i\subset A_i$ for a collection of finite sets $A_i\subset M$ are {\\it compatible}, if there exists a linear function $\gamma:M\to\Z$, whose restriction to $A_i$ attains its maximal value exactly at the points of the face $B_i$ for every $i$ (or, equivalently, if $\sum_i B_i$ is a face of $\sum_i A_i$).

\begin{defin}\label{defblinder}
1. For $b\in\Z$, denote $\bar A=A+IA$ and $\check A:=\bar A\setminus\{t=0\}$.

2. An edge $E\subset A$ is called a {\it blinder} if $\bar E$ is an edge of $\bar A$. 

3. The blinder $E\subset A$ is compatible with a unique {\it symmetric face} $\check E\subset\check A$ such that $\check E=I\check E$.

4. Choose the linear surjection $\gamma:M\to\Z$ such that 
$\gamma(\check E)=\max\gamma(\check A)$. 

The {\it multiplicity} $h_E\geq 1$ of the blinder $E$ is the difference 
$\max\gamma\bar A-\max\gamma\check A$.

5. The {\it projection along} $E$ is an affine surjection $p_E:\Z^3\to\Z^2$ sending $\bar E$ to 0. 

6. The {\it link polygon} $P_E$ is the nonconvex lattice polygon $(\R_+\cdot p_E\bar A)\setminus(\R_+\cdot p_E\check A)$.
\end{defin}
\begin{rem}\label{remsuppgeom}
1. The symmetric face and multiplicity are defined unless $A$ is of type D (in the latter case it is convenient to define $\check E:=\varnothing$ and thus $h_E=+\infty$).

2. A blinder $E$ is contained in the plane  $\{t=b\}$ for some $b\in\Z$. Its every compatible face $F\subset \check A$ is the Minkowski sum of $E$ and the face $F-E$ of the set $A_b\cup (I(A_b)-b\cdot m)$, for $A_b:=A\setminus\{t=b\}$. 

\end{rem}

\begin{defin}\label{deftrop}
1) The tropical fan $[A]$ consists of the exterior normal cones to the edges of $A$ with weights equal to the lattice length of the edge.

2) For $A$ and $B\subset\Z^3$, the intersection product $[A]\cdot[B]$ consists of exterior normal rays to the facets $P+Q\subset A+B$ with weights equal to $\MV(P,Q)$ (where $P\subset A$ and $Q\subset B$ are the corresponding compatible faces, and their mixed volume is taken in the affine span of $P+Q$).
\end{defin}
\begin{exa}
If $A$ is a segment of lattice length $d$, then $[A]$ is its normal hyperplane $H$ with weight $d$, and $[A]\cdot[B]$ equals $d\cdot [$projection of $B$ along $A]$ in $H$. 
\end{exa}

{\bf Genericity conditions} are obtained in Theorem 5.6, \cite{crit}, for a symmetric complete intersection $f(x_1,x_2,x_3,\ldots,x_n)=f(x_2,x_1,x_3,\ldots,x_n)=0$ for any $n$. We cite them here, because they significantly simplify for $n=3$, and this simplified version is used to prove classification results of the preceding sections. Let $0\in A\subset M\simeq \Z^3$.

\begin{defin}\label{defInondeg}
The polynomial $f\in\C^A$ and the curve $f=f\circ I=0$ are called {\it $I$-nondegenerate} (to be distinguished from nondegeneracy, see e.g. \cite{crit}, Definition 1.1), if:

1) for compatible non-blinder faces $P\subset A$ and $Q\subset IA$, the system of equations $f|_P(u,v,w)=f|_Q(v,u,w)=0$ is regular (i.e. all of its roots are regular);

2) for a blinder $E\subset A$ and a compatible face $F\subset\check A$, the system of equations $f|_E(u,v,w)=[f(u,v,w)-(u/v)^bf(v,u,w)]|_{F-E}=0$ is regular (defining $d$ as in Remark \ref{remsuppgeom}.2);

3) the system of equations $f(u,v,w)=\frac{f(u,v,w)-f(v,u,w)}{u^d-v^d}=u^d-v^d=0$ is regular in $\CC^3$ (defining $d$ as in Observation \ref{ex:diagcomponents&singularities}).

\end{defin}
\begin{rem} 
The expression in the square bracket is supported at the last set in Remark \ref{remsuppgeom}.2. The following is the summary of Sections 5.3 and 3.3 in \cite{crit} for $n=3$.

\end{rem}


\begin{theor}\label{thsingnum}
1. All polynomials $f\in\C^A\setminus($an algebraic hypersurace$)$ are $I$-nondegenerate, for all of them the proper part of the symmetric curve $\{f(u,v,w)=f(v,u,w)=0\}$ is smooth and has the same topological type.

2. It intersects a generic plane of the form $\{u=cv\}$, including each of its $d$ diagonal components, 
at $\sharp A$ many points, where $\sharp A$ is defined as
$$MV(\check A/m,A/m)-\Vol_\Z (A/m)=\Vol_\Z (A/m)-\sum_{\mbox{ \scriptsize \rm blinder } E\subset A} h_E\cdot \Vol_\Z E=$$ $$=\Vol_\Z (A/m)-\sum_{b\in\Z}\bigl(\Vol_\Z (A/m)-\MV(A/m,A_b/m)\bigr).\eqno{(*)}$$

3. Its Euler characteristics equals $d\cdot\sharp A$ minus $$\MV(\check A,\check A,A)+\MV(\check A,A,A)-d\cdot\MV(\check A/m,A/m)=\Vol_\Z\bar A-\sum_{\mbox{ \scriptsize \rm blinder }E\subset A}\Vol_\Z (E)\cdot\Vol_\Z P_E.$$ 

4. It is irreducible at least if $A+\{0,d\cdot m\}$ can be shifted to the interior of $\check A$.

5. Its tropical fan 
equals $[\check A]\cdot [A]-[A]\cdot [A]-[A]\cdot D$, where $D$ is the diagonal plane with weight $d$ (this holds even if we drop assumption (3) in the definition of $I$-nondegeneracy).

6. The closure of the symmetric curve in the $\check A$-toric variety intersects its orbits transversally (in particular, it is smooth at every point of intersection with a 2-dimensional orbit, and avoids smaller orbits).
\end{theor}
\begin{rem}
1. In particular, this gives the genus of the proper part: for any sch\"on curve $C\subset\CC^n$, we have $2-2g(C)=e(C)+($the total multiplicity of the rays of $\Trop C)$. The right hand side terms in this equality are known from (3) and (5) above.

2. Part (1) implies that whatever we prove about the number of components or other topological invariants of symmetric curves for generic $f\in\C^A$ (i.e. for all $f$ in a non-empty Zariski open $U\subset\C^A$), this holds in particular  for all $I$-nondegenerate $f$. This is because the Zariski open sets $U$ and $V:=\{I$-nondegenerate polynomials$\}$ overlap, and any topological invariant of the symmetric curve is constant as a function on $V$ according to part (1).
\end{rem}

\begin{exa}\label{exaquadric}

To illustrate the above theorem and subsequent results, it may be helpful to keep track of the example of a generic quadratic form $f\in\C^A,\,A=\{\ru\geqslant 0,\rv\geqslant 0,\rw\geqslant 0,\ru+\rv+\rw\leqslant 2\}$. In this case, $A$ has a blinder edge $E$ at the $\rw$-coordinate axis, and the symmetric curve $f=g=0$ is a union of two plane conics: one is the diagonal part in the plane $u=v$, and the other one (in a perpendicular plane) is the proper part. 
The two conics intersect each other in $$\Vol_\Z (A/m)-h_E\cdot \Vol_\Z(E)=4-1\cdot 2=\MV(A/m,A_0/m)=2$$ points, because $A_b/m=A/m$ unless $b=0$, so the terms for $b\ne 0$ vanish in Theorem \ref{thsingnum}.2.
\end{exa}

\section{Connectivity and singularities of symmetric curves: the proofs}\label{sconn}

In this section and the next one,  we prove theorems of Section \ref{sre}.

\subsection{Singularities do exist} 
\begin{lemma}\label{lposit}
The number of intersection points $\sharp A$ in the setting of Theorem \ref{thsingnum} is strictly positive unless $A$ is of type D, E, C1, C2 or I10.
\end{lemma}
{\it Proof.} We recall that the main ingredient in the quantity $\sharp A$ is $A\mapsto A/m=:B$, the image of $A$ under the projection $M\to M/M^{-I}$. The image of a blinder edge $E\subset A$ under this projection will be called a blinder edge of $\conv (B)$.

In order to prove that $\sharp A$ is positive, we shall construct non-overlapping pieces of $\conv (B)$, whose area estimates from above the subtrahend in $(*)$. 

In order to construct such pieces, denote by $B_1,\ldots,B_q$ the connected components of the union of the blinder edges of $\conv(B)$, and consider three cases: $q>2$, $q=2$ and $q=1$ (note that for $q=0$ the sought number $\sharp A$ is positive by definition under our assumptions on $A$). 

{\bf a)} $q>2$. Let $V_i\in B\setminus B_i$ be the  boundary point of $B$ closest counterclockwise to the polygonal chain $B_i$ (see Figure \ref{fig:proof4}).

Note that the preimage of $B_i\cap B$ under the projection $A\to B$ is contained in a plane parallel to $M^I$, while the preimage of $V_i$ is not entirely contained in this plane. Thus $h_E\cdot\Vol_\Z (E)$ for a blinder $E\subset B_i$ is at most the lattice area of the triangle with base $E$ and vertex $V_i$.

Then $\sharp A$ is at least the lattice area of $B$ minus the lattice areas of the triangles $\conv(B_i\cup\{V_i\})$. The interior of latter triangles do not overlap and do not cover the interior of $\conv (B)$. Thus the number of singularities is the lattice area of a lattice polygon with non-empty interior. It is therefore strictly positive.

{\bf b)} If $q=1$, then the preimage of $B_1$ under the projection $A\to B$ is contained in one plane $L$ parallel to $M^I$, and the image of $A\setminus L$ under this projection will be denoted by $B'\subset B$  (see Figure \ref{fig:proof4}). For every blinder $E\subset B_1$, we can choose a triangle with the base $E$, one more vertex in $B'$, no other points of $B'$ in it, and the area at least $h_E\cdot \Vol_\Z (E)$, so that these triangles do not overlap. These triangles do not cover $\conv (B)$, because $B'$ contains more than one point (otherwise $A$ would have exceptional type $C2$). Thus the difference of the lattice area of $B$ and the lattice area of the union of these triangles does not exceed $\sharp A$ and exceeds 0.

{\bf c)} The case $q=2$ can be resolved similarly with a mixture of observations from (a) and (b): in this case, the number of singularities is positive, unless each of $B_1$ and $B_2$ consists of one blinder, these blinders are parallel, and together they cover $B$. In the latter case, $A$ is contained in a pair of lines parallel to each other and to the plane $M^I$, i.e. has type I10.
\hfill$\square$

\begin{figure}
    \centering
    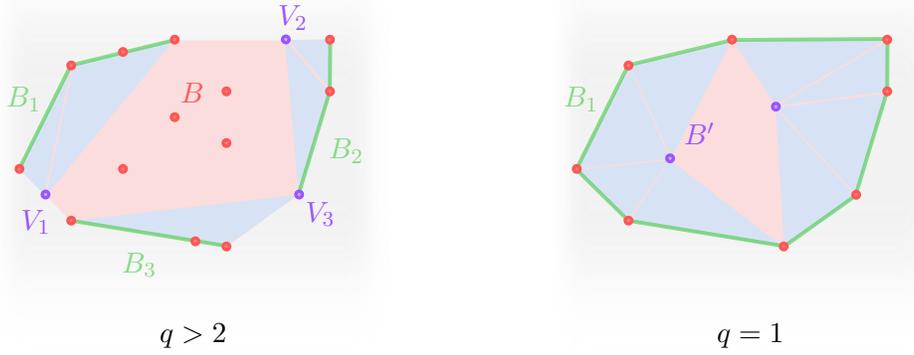
    \caption{The proof of Lemma \ref{lposit}}
    \label{fig:proof4}
\end{figure}

\subsection{Toric spans of proper components}
\begin{defin}
The toric span of a connected subset $S$ of the algebraic torus $T$ is the smallest subtorus $G\subset T$ such that $g\cdot S\subset G$ for some $g\in T$.

\end{defin}

Note that such minimal subtorus always exists: it is the intersection of all algebraic subgroups containing $S$ up to a shift, and this subgroup is connected, because $S$ is connected.

\begin{utver}\label{psame}
1) All proper irreducible components of a generic non-exceptional symmetric curve have the same toric span. 

2) The number of intersection points of a proper component with a diagonal component of a generic non-exceptional  symmetric curve is positive and does not depend on the choice of the components. 
\end{utver}
The proof will require the following observation.
\begin{lemma}\label{lminors}
Unless the support set $A$ is of type D or a point $x\in T$ satisfies $m^d(x)=1$, with $d\geq 1$ defined as in Observation \ref{ex:diagcomponents&singularities}, the equations $$f(x)=f(I(x))=0$$
(regarded as linear equations on $f\in\C^A$ for the fixed point $x$) are linearly independent.
\end{lemma}
\begin{proof}
We work in the standard coordinates $x=(u,v,w)$ in the algebraic torus, so that $I(u,v,w)=(v,u,w)$, and $m(x)=u/v$. Since the statement is invariant under shifts of $A$, assume w.l.o.g that $0\in A$ and assume to the contrary that every $2\times 2$ minor of the above system of linear equations vanishes. 

For the minor corresponding to the constant coefficient of $f$ and its coefficient at another monomial $\mu\in A$, this implies $\mu(x)/\mu(I(x))=1$. The latter quotient has the form $(u/v)^{d_\mu}$ for some $d_\mu\in\Z$. Since $\GCD(d_\mu,\, \mu\in A)=d$ for $d_\mu$ so defined, the equalities $(u/v)^{d_\mu}=1$ imply $(u/v)^{d}=m^d(x)=1$.
\end{proof}

{\it Proof of part 1 of Proposition \ref{psame}.} For a subtorus $G\subset T$ and a polynomial $f\in\C^A$, let $G_f$ be the union of all irreducible components of the curve $\{f(x)=f(I(x))=0\}$ having toric span $G$. Let $G_A$ be the closure of the union $\bigcup_{f\in\C^A} \{f\}\times G_f$ in the product $\C^A\times\{m^d\ne 1\}$. Consider the inclusions:

$$G_A \subset \{(f,x)\,|\,f(x)=f(I(x))=0\}\subset\C^A\times\{m^d\ne 1\}$$

(1) If $G_f\ne\varnothing$ for generic $f\in\C^A$, then the set on the left has codimension 2 in the set on the right, at every point $(f,x)$ for generic $f$.

(2) The projection of the set in the middle to $\{m^d\ne 1\}$ is a vector bundle, whose fiber has codimension 2 in $\C^A$ (by Lemma \ref{lminors}), thus the set in the middle is itself irreducible of codimension 2 in the set on the right.

From (1) and (2), we deduce that the set on the left has a codimension 2 component and is contained in the middle set, which is irreducible  and of codimension 2. Thus the two sets are equal: if a generic symmetric curve has a proper component with a toric span $G$, then all of its proper components have toric span $G$. \hfill$\square$

\begin{rem}
Notice how this reasoning fails if we try to extend it to diagonal components: in order to do so, we should substitute the set $\{m^d\ne 1\}$ with the whole torus $T$. In this case, looking at the dispaly formula in the proof, the projection of the middle set to $T$ would have fibers of codimension 2 over $\{m^d\ne 1\}$ and codimension 1 over $\{m^d= 1\}$, failing to be a vector bundle.
\end{rem}

{\it Proof of part 2.} Similarly to part 1, for a given diagonal component $H\subset\{m^d=1\}$, integer $k\in\Z$ and $f\in\C^A$, let $k_f$ be the union of all irreducible components of the symmetric curve $\{f(x)=f(I(x))=0\}$ intersecting $H$ at $k$ points, and let $k_A$ be the closure of the union $\bigcup_{f\in\C^A} \{f\}\times k_f$ in the product $\C^A\times\{m^d\ne 1\}$. Consider the inclusions:

$$k_A \subset \{(f,x)\,|\,f(x)=f(I(x))=0\}\subset\C^A\times\{m^d\ne 1\}$$

Again, if $k_f\ne\varnothing$ for generic $f\in\C^A$, then the set on the left has codimension 2 in the set on the right at its generic point. Since the set in the middle is irreducible of codimension 2 (see part 1), it equals the set on the left: if a generic symmetric curve has a proper component intersecting the diagonal component $H$ at $k$ points, then all of its proper components intersect $H$ at $k$ points.

Thus the proper part of the symmetric curve intersects the diagonal component $H$ at $k\times($the number of proper components$)$ points. By Theorem \ref{thsingnum}, the latter quantity is independent on the choice of $H$, and by Lemma \ref{lposit} it is positive, thus every proper component intersects every diagonal component at the same positive number of points.
\hfill$\square$\\

\begin{sledst}\label{cpropersymm}
A proper component of a generic non-exceptional symmetric curve is itself symmetric with respect to the involution $I$.
\end{sledst}
\begin{proof}
Assume that a proper component $C_1$ is symmetric to another proper component $C_2$. By Proposition \ref{psame}, $C_1$ intersects the diagonal $m=1$ at some point $x$. Since $x\in C_2$ as well by symmetry, the proper part of the symmetric curve is not locally irreducible at $x$, and hence is not smooth, which contradicts Theorem \ref{thsingnum}. 
\end{proof}

\subsection{Planar components}
Our next goal is to understand for which $A$ a proper component of the generic symmetric curve is {\it planar}, i.e. its toric span is strictly smaller than $T$. 

\begin{lemma}\label{lsymnondeg}
Given a monomial $\mu\in M^I$ and an $I$-nondegenerate polynomial $f\in\C^A$, for generic $c\in\C$ the system of equations $f(x)=f(I(x))=\mu(x)-c=0$ is nondegenerate (see e.g. Definition 1.1 in \cite{crit}).

\end{lemma}
\begin{rem}
We shall see from the proof that the system is nondegenerate at infinity (i.e. the condition of Definition 1.1 in \cite{crit} is satisfied for all non-zero $l$) once $f$ satisfies a weaker assumption: 
for every blinder $E\subset A$, the polynomial $f|_E$ is not identical zero, and, for every pair of compatible non-blinders $P\subset A$ and $Q\subset IA$, the system $f|_Pf\circ I|_Q=0$ defines a complete intersection in $T\simeq\CC^3$.
\end{rem}
\begin{proof} Denote the system that we study by $f=g=h=0$, and the face of $A$ at which a liner function $l:M\to\R$ attains its maximum, by $A^l$. We want to prove that, for every non-zero $l\in M^*$, the system $f^l=g^l=h^l=0$ has no roots in $T$ (see Definition 1.1 in \cite{crit} for notation).

For any $l\in M^*$ such that $l\cdot\mu\ne 0$, the polynomial $h^l$ is a monomial, thus has no roots in $T$.

For any $l\in M^*$ such that $l\cdot\mu=f\cdot m=0$, and $A^l$ is a blinder, $g^l=h^l=0$ has no roots for generic $c\in\C$, as the Newton polytopes of $g^l$ and $h^l$ are parallel segments. 

For any other non-zero $l\in M^*$, the polynomials $f^l=f|_{A^l}$ and $g^l=g|_{(IA)^l}$ define a complete intersection, thus so do $f^l=g^l=h^l=0$. 
\end{proof}

\begin{rem}\label{remsymnondegmv}
1. In the setting of Lemma \ref{lsymnondeg}, the number of solutions of the system equals the mixed volume of the support sets (by the Bernstein formula, see \cite{bernst} or e.g. Sections 2.4-2.5 in \cite{crit}). This in turn equals the mixed area of the projections $A/\mu$ and $I(A)/\mu$ (e.g. by Remark 2.20 in \cite{crit}), as soon as $\mu\in M$ is primitive. 

2. How many of these solutions belong to the diagonal components of the symmetric curve? To find this, we should solve the system $f(x)=m(x)^d-1=\mu(x)-c=0$. It is nondegenerate for generic $f$ and $c$ by the same considerations, 
so the number of roots equals the mixed volume of $A$ and the support segments of the two binomials. This equals $$d\cdot h,$$ where $h$ is the lattice length of the image of the convex hull of $A$ under its projection $M\to M/L\simeq\Z^1$ along the plane $L$ spanned by $m$ and $\mu\in M$. 

3. From (1) and (2), we deduce the inequality $MV(A/\mu,I(A)/\mu)\geqslant d\cdot h$. We now classify for what $A$ it becomes an equality (or, equivalently, $\{\mu=c\}$ does not intersect the proper part of the symmetric curve for generic $c$, i.e. $\mu$ is constant on every proper component).
\end{rem}

\begin{lemma}\label{lmvcompare}
Let $J:\Z^2\to\Z^2$ be the involution $(x,y)\mapsto(x,-y)$. Let $P$ be a lattice polygon, and $L(P)$ the length of its  projection to the horizontal axis. Then we have
$$MV(P,J(P))>2L(P),$$
except for the following cases (see Figure \ref{fig:lemma}):

-- $P$ is contained in a vertical line;

-- $P$ is up to a shift contained in the stripe $0\leqslant y\leqslant 1$;

-- $P\cap\Z^2$ is up to a shift contained in the set $\{y=0\}\cup\{(0,1),(0,-1)\}$.
\end{lemma}
One can already recognize that the three exceptional cases of this lemma are the shadows of exceptional support sets of types C1, I1 and I2.

\begin{figure}[h]
    \centering
    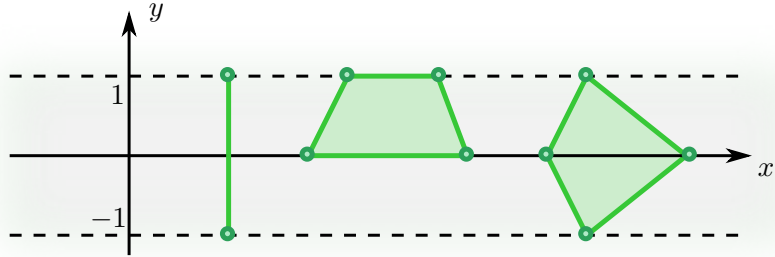
    \caption{The exceptional cases in Lemma \ref{lmvcompare}}
    \label{fig:lemma}
\end{figure}

\begin{proof} The sought inequality can be directly verified if $P$ has one of the following four types up to reflections w.r.t. the coordinate axes:

1. A triangle with vertices $(0,0),(a,0),(b,c),\, a\geqslant b\geqslant 0,\, |c|>1$; 

2. A triangle with vertices $(0,0),(a,1),(b,-1),\,a>b\geqslant 0$; 

3. A quadrilateral with vertices $(0,0),(a,1),(b,-1),(c,0),\,c\geqslant a>b\geqslant 0$;

4. A segment with end points $(0,0)$ and $(a,b),\, a>0, b\geqslant 2$.

Now, coming back to arbitrary $P$, let $S$ be a segment connecting one of its leftmost vertices to one of its rightmost vertices. Taking the convex hull of $S$ and one or two more points of $P$, we can find a lattice polygon $P'$ of type (1-4) so that $S\subset P'\subset P$, unless $P$ is of one of the exceptional types from the statement of the lemma. Notice that $L(S)\leqslant L(P')\leqslant L(P)=L(S)$ implies $L(P')=L(P)$.
Now we have (by the monotonicity of the mixed volume): $$\MV(P,J(P))\geqslant\MV(P',J(P'))>2L(P')=2L(P).$$
\end{proof}

\begin{utver}\label{pflat}
If the toric span of some irreducible component of the proper part of a generic symmetric curve supported at $A$ is not the full torus, then $A$ is exceptional.

\end{utver}
{\it Proof.} If the toric span $T'$ of some (hence every by Proposition \ref{psame}) proper irreducible component of a generic symmetric curve is not the full torus, then it is  contained in the subtorus $m'=1$ for some primitive monomial $m'$.

First assume $m'\in M^{-I}$, i.e. $m'=-m$, then the shifted subtorus $\{m'(x)=c\}$ does not intersect the diagonal $\{m^d(x)=1\}$ (unless it coincides with one of the components of the diagonal). This implies that the proper component with the toric span $\{m'(x)=c\}$ does not intersect the diagonal components, thus $A$ is exceptional by Proposition \ref{psame}.2.

So from now on we may assume $m'\notin M^{-I}$, and, as a consequence, $m'$ and $I(m')$ generate a sublattice $L$ (of dimension 1 or 2) that non-trivially intersects the plane $M^I$.

Thus we may moreover assume with no loss of generality that $m'\in M^I$: Corollary \ref{cpropersymm} ensures that the toric span $T'$ belongs both to the subtorus $m'=1$ and $I(m')=1$. In this case, changing $m'$ to a primitive element in the (non-trivial) intersection of the plane $M^I$ and $L$, the toric span $T'$ still belongs to the subtorus $m'=1$. 

By Proposition \ref{psame}, we now see that every proper component is contained in a shifted torus of the form $m'=c_i$. Thus a shifted torus $m'=c$ for generic $c\ne c_i$, intersects the symmetric curve only at its diagonal points.

In particular, the number of solutions $\sharp_1$ of the system $f(x)=f(I(x))=m'(x)-c=0$ for generic $c$ equals the number of solutions $\sharp_2$ of the system $f(x)=m(x)^d-1=m'(x)-c=0$.

The numbers $\sharp_1$ and $\sharp_2$ are computed in Remark \ref{remsymnondegmv} and can be compared by Lemma \ref{lmvcompare}: define the linear map $\Phi:M\to\R^2$ by sending the monomials $u,v$ and $m'$ to $(1,1/d), (1,-1/d)$ and $(0,0)$ respectively, and apply  Lemma \ref{lmvcompare} to the convex hull $P$ of the image of $A$ (note that,  by the definition of the denominator $d$, the polygon $P$ is a lattice polygon as soon as $A$ is shifted to contain 0). We shall see that $$\sharp_1=\MV(A/m',I(A)/m')=\begin{vmatrix}
1 & -1/d \\
1 & 1/d
\end{vmatrix}\cdot\MV(P,J(P))>d\cdot L(P)=d\cdot h=\sharp_2,$$ unless $P$ is one of the exceptions in Lemma \ref{lmvcompare}.

Since we know that $\sharp_1=\sharp_2$, we conclude that $P$ is one of the exceptions in Lemma \ref{lmvcompare}. This implies that $A$ is of the exceptional type $E,D,C_1,I_1$ or $I_2$.

\hfill$\square$

\subsection{Proof of Theorems \ref{thdim}--\ref{thintroirr}(2) for $I$-nondegenerate $f\in\C^A$}\label{ssproofs1} 

For the exceptional types of $A$, to which these theorems apply, their statements can be directly verified (see Remarks \ref{exadisconn} and \ref{exaI}). So we assume that $A$ is not exceptional.

By Definition \ref{defInondeg}, the symmetric curve $f(x)=f(I(x))=0$ is a union of two transversal smooth curves: the diagonal part $f(u,v,w)=u^d-v^d=0$ and the proper part (the rest). 

By Lemma \ref{lposit} and Proposition \ref{psame}, every proper component intersects every diagonal component by the same positive number of points, hence their union is connected, and neither of the two parts is empty.
By Proposition \ref{pflat}, no component of  the proper part is contained in a proper subtorus of $T$.
\hfill\l{$\square$}

\section{Irreducibility of the proper part of a symmetric curve
}\label{sirredproof}

\subsection{The plan} Our strategy to prove Theorem \ref{thintroirr}(1) will be as follows:

1) Classify all minimal (by inclusion) non-exceptional support sets;

2) Collect a bag of tricks to prove irreducibility for each of them;

3) Deduce irreducibility for all the other non-exceptional support sets inductively, starting from a minimal non-exceptional support set, and adding further lattice points one by one.

\begin{rem}
This approach to the proof requires a significant amount of serendipity: one should ``guess from nowhere'' the list (E) of all exceptional support sets on one hand, and the list (M) of all minimal non-exceptional sets on the other hand. Once we have both of them, it is relatively straightforward to check the three facts comprising the base of the induction in the step 3 of the plan above:

-- for every support set in (E), the symmetric curve has a reducible proper part;

-- for every set from (M), the symmetric curve has an irreducible proper part;

-- every set outside (E) contains a set from (M).

Of course, the dearest wish here is to invent a proof that would not require serendipity and produce the list (E) (and (M) if necessary) automatically. 
Of them, the list $(E)$ can be guessed beforehand from Remark \ref{bigq00}. So the key problem is to invent a proof that would not require the list (M) at all, or invent a way to guess the list (M) beforehand as well.
\end{rem}

The first step of the plan will require the discussion of concrete configurations of lattice points in a special position with respect to the involution $I$. To facilitate this discussion, we shall use the notation that we allow ourselves to introduce by an example rather than by a (tedious) general description. By writing that the finite lattice set $A=\{b_0,\ldots,b_4\}\subset M$ has the shape
$$
\begin{array}{c c c c|c c c}
   b_0 & & & &     e_0  & & \\
   b_1 & & & &     e_1  & * & \\
   & b_2 & & &     e_2  & * & * \\
   & & b_3 & b_4 & e_3  & * & * \\   
   \midrule
   c_1 & c_2 & c_3 & c_4 &  2 & 1 & 1 \\
\end{array}
$$
we mean the following:

-- under the projection $M\to M/M^I$, the points $b_0$ and $b_1$ go to $c_1$, and $b_i$ to $c_i$ for $i>1$;

-- the order of the images $c_1,\ldots,c_4$ in the line $M/M^I$ is as in the bottom row of the table (and in particular $c_i\ne c_j$ for $i\ne j$; we do not care about the orientation of $M/M^I$ and the direction in which we read the bottom line);

-- under the projection $M\to M/M^{-I}$, the points $b_3$ and $b_4$ go to $e_3$, and $b_i$ to $e_i$ for $i<3$;

-- the last column says that the affine span of $e_2$ and $e_3$ is one-dimensional (i.e. $e_2\ne e_3$), the second last column says that the affine span of $e_1,e_2,e_3$ is one-dimensional as well (i.e. $e_1$ is on the line through $e_2$ and $e_3$), and the third last column says that  the affine span of $e_0,e_1,e_2,e_3$ is two-dimensional (i.e. $e_0$ is not on the line through $e_2$ and $e_3$).

\vspace{1ex}

Note that we allow any affine dependencies between $e_i$'s that are not prohibited by the last columns of the table, e.g. $e_1$ may coincide with $e_2$ or $e_3$ in our example.

If we do not want to give names to the projected points, then we write $*$ instead of $c_i$'s and $e_j$'s in the table.

\begin{rem}\label{kh15doesnothelp}
Before proceeding to our plan to prove irreducibility, we briefly discuss how known methods fail to work here. Most notably,  \cite{kh15} claims that the algebraic set $V\subset\CC^n$ given by the nondegenerate equations $f_1=\cdots=f_k=0$ supported at $A_1,\ldots,A_k\subset\Z^n$ is irreducible, unless $m$ of the $A_i's$ can be shifted to the same $m$-dimensional sublattice (for an explanation of this obvious exception, see e.g. Examples 2.16-2.17 in \cite{crit}).

We could try to apply this to the study of symmetric curves in the following two ways.

1) Apply \cite{kh15} to the system of equations $f=f\circ I=0$. This is impossible, because the symmetric curve is singular in most cases (Theorem \ref{thsing}), rendering the system $f=f\circ I=0$ degenerate.

2) Assuming $0\in A$ with no loss of generality, apply \cite{kh15} to the system $$f=g=0,\quad g(u,v,w):=\frac{f(u,v,w)-f(v,u,w)}{u^d-v^d}=0,$$ which describes the symmetric part of the curve. This is impossible too: once the support set $A$ has an edge $E$ disjoint from the plane $M^I$, the support set of $g$ contains the segment $E':=E-(d,d,0)$. In many cases (including whenever $E$ is a blinder), $E$ and $E'$ are compatible edges of the support sets of $f$ and $g$, and $f|_E$ coincides with $g|_{E'}$ up to a monomial multiplier.

In Section \ref{snext} we discuss possible approaches to adapting the classical technique of \cite{kh15} to symmetric complete intersections; in this section we stick to the plan presented at the beginning.
\end{rem}
\begin{rem}\label{rem:kho}
It is never mentioned in \cite{kh15} that the nondegeneracy of the complete intersection 
is a sufficient genericity assumption for the irreducibility theorem 17 therein. One can derive it carefully examining the proofs (starting from the proof of Lemma 28). However, it is easier to see that the nondegeneracy implies the irreducibility from the facts that

1) all nondegenerate complete intersections are smooth and homeomorphic to each other (see e.g. Corollary 2.16 in \cite{schon}), and

2) almost all of them are connected by \cite[Theorem 17]{kh15} (because the connectedness is equivalent to the irreducibility for smooth varieties).
\end{rem}

\subsection{Inductive irreducibility}
We start with a statement that connectivity of a generic curve or irreducibility of its proper part for a certain support set $A$ implies irreducibility for the larger support set $A\cup\{a\}$ under some conditions on $A$ and $a$.
On one hand, it will offer an inductive step for stage (3) of our plan, and on the other hand it will offer one of the tricks for stage (2), because it allows to deduce irreducibility for some minimal non-exceptional support sets from connectivity of exceptional ones.

\begin{utver}[\bf Inductive irreducibility] \label{propind}
Consider support sets $A\subset A'\subset M$ such that $A'=A\sqcup \{a\}$ is not exceptional, and their convex hulls have the same projection along $M^I$. Denote $d:=d(A)$ and $d':=d(A')|d$ the respective denominators (see Definition \ref{def:diag&properparts}).
Assume one of the two:

1) $d'=d$, and the proper part of the generic symmetric curve supported at $A$ is irreducible and non-empty;

2) $d'<d$, and the generic symmetric curve supported at $A$ is connected (i.e. not of type D, E, C1, C2 or I10 according to Theorem \ref{thconn}).

Then the proper part of the generic symmetric curve supported at $A'$ is irreducible.
\end{utver}
{\it Proof of Part 1.}
From the assumption, we can choose a primitive monomial $m'\in M^I$ such that the convex hulls of $A$ and $A'$ have the same projections along the line generated by $m'$. Indeed, observe that $A'/M^I$ has dimension 1 since it is not of type $D$. If $\conv(A')$ has dimension at most $2$, any $m'$ ortogonal to it will do. If $\conv(A')$ has dimension $3$, any $m'$ will do if $a$ is not a vertex of $\conv(A')$, otherwise take $m'$ whose scalar product with $A'$ has a strict maximum at $a$.

By Proposition \ref{pflat}, since $A'$ is not exceptional, a generically shifted torus $m'=c$ transversally intersects every proper irreducible component of the generic symmetric curve supported at $A'$ at some point $x_i$. So it is enough to prove that the points $x_i$ can be connected with a path in the smooth part of this curve. 

In order to construct such a path, choose a generic polynomial $f\in\C^A$ and generic $c\in\C$, and consider the perturbation $f_\varepsilon(x)=f(x)+\varepsilon a(x)\in\C^{A'}$. By Remark \ref{remsymnondegmv}, the systems $$f_\varepsilon(x)=f_\varepsilon(I(x))=m'(x)-c=0\eqno{(*_\varepsilon)}$$ have the same number of solutions $N$ for all small $\varepsilon$ (including $\varepsilon=0$) and among them the same number of solutions $N'\leqslant N$ belong to the diagonal part of the curve $$f_\varepsilon(x)=f_\varepsilon(I(x))=0.\eqno{(**_\varepsilon)}$$

Since the roots of the system $(*_\varepsilon)$ are BKK-nondegenerate by Lemma \ref{lsymnondeg}, we have germs of analytic curves $$x_i:(\C,0)\to T,\,\varepsilon\mapsto x_i(\varepsilon),\,i=1,\ldots,N,$$
such that $x_i(\varepsilon),\,i=1,\ldots,N$, are the roots of the system $(*_\varepsilon)$, and those with $i=1,\ldots,N'$ belong to the diagonal part of the curve $(**_\varepsilon)$.

Any roots $x_i(0)$ and $x_j(0)$ in the proper part of the curve $(**_0)$ (i.e. for $N'<i<j\leqslant N$) can be connected to each other with a path $\gamma_0$ in the smooth part of this curve, because its proper part is irreducible by assumption. Now the crucial observation is that for a small $\varepsilon$ we can construct the perturbation $\gamma_\varepsilon$ of $\gamma_0$ connecting $x_i(\varepsilon)$ to $x_j(\varepsilon)$ in the smooth part of the curve $(**_\varepsilon)$.

{\it Proof of Part 2.} The difference with the part 1 is that we may be unable to connect $x_i(0)$ and $x_i(0)$ in the smooth part of the curve $(**_0)$. However, in this case it is enough to connect them via the set $S:=\{f_0(x)=f_0(I(x))=0\}\setminus\{m^{d'}=1\}$ with a path that we shall denote by $\gamma_0$. This is possible, because $d'<d$ ensures that $S$ contains some diagonal component of the curve $(**_0)$, and every its diagonal component intersects every its proper component by Proposition \ref{psame}.2. 
If the resulting path $\gamma_0$ passes through singular points $z_k\in S$, this should not frighten us, because these singularities will smooth out (with connected Milnor fibers $M_k$) as we perturb $\varepsilon$. Splitting the path $\gamma_0$ into pieces $\gamma_0^k$ at the singular points $z_k$, we can perturb these pieces into the paths $\gamma_\varepsilon^k$, connecting the Milnor fibers $M_k$ in the smooth part of the curve $(**_\varepsilon)$. Then we can connect the paths $\gamma_\varepsilon^k$ via the connected Milnor fibers $M_k$ into one path $\gamma_\varepsilon$ connecting $x_i(\varepsilon)$ and $x_i(\varepsilon)$ in the smooth part of the curve $(**_\varepsilon)$.
\hfill$\square$

\subsection{Minimal non-exceptional support sets} For every such $A$, one might hope to deduce the irreducibility of the proper part of the generic symmetric curve $f(x)=f(I(x))=0,\,f\in\C^A$, by Inductive irreducibility as follows. 
\begin{exa}\label{ex:skew}{\bf (The skew tetrahedron, Figure \ref{fig:minimal})} Let $A'$ have the shape
$$
\begin{array}{c c c c|c c c}
   b_1 & & & &     e_1  & * & * \\
   & b_2 & & &     e_2  & * & \\
   & & b_3 & &     e_3  & & \\
   & & & b_4 &     e_4  & * & * \\   
   \midrule
   c_1 & c_2 & \frac{c_1+c_4}{2} & c_4 &  2 & 1 & 1 \\
\end{array}
$$

Then the corresponding symmetric curve has irreducible proper part by Proposition \ref{propind}.2, applied to $A=\{b_1,b_3,b_4\}\subset A'$. Indeed, it is clear from the $c_i$'s that $A$ cannot be of type D, E and I1. The first two columns on the right guarantee that $b_3$ does not belong to the plane parallel to $M^{-I}$ containing $b_1$ and $b_4$, that is $A$ is not of type $C_1$. Considering the last column, we deduce that any line parallel to $M^{-I}$ contains at most one point in $A$ and that $A$ is not of type C2. Eventually, we can assume, up to a shift, that $c_1=0$. Thus $d=\vert c_4/2\vert$, $d'=\GCD(d,c_2)<d$, and $A$ is of type $I2$, so Proposition \ref{propind}.2 is applicable.
\end{exa}
However, some non-exceptional support sets $A'$ cannot be split into $A\cup\{a\}$ fitting the assumptions of the Inductive irreducibility.
\begin{exa}\label{ex:high}{\bf (The high tetrahedron, Figure \ref{fig:minimal})}
Let $A'$ have the shape
$$
\begin{array}{c c c c|c}
   b_1 & & & b_4 &     e_1  \\
   & b_2 & & &     e_2  \\
   & & b_3 & &     e_3 \\
   \midrule
   c_1 & c_2 & c_3 & c_4 &  2 \\
\end{array}
$$

Then there is no three-point subset $A\subset A'$ to which the Inductive irreducibility would be applicable. Indeed, removing either $b_2$ or $b_3$ 
leads to a support $A$ of type C1. In particular Proposition \ref{propind}.2 never applies. According to Remark
\ref{exadisconn}, 
the proper part of the generic symmetric curve supported on $A$ is reducible in general so that Proposition \ref{propind}.1 does not apply either. Removing either $b_0$ or $b_3$ changes the projection along $M^I$ and is therefore not allowed.
\end{exa}
\begin{exa}\label{ex:low}{\bf (The low tetrahedron, Figure \ref{fig:minimal})} Let $A'$ 
have the shape
$$
\begin{array}{c c c|c c c}
   b_1 & & &     e_1  & * & \\
   b_2 & & &     e_2  & * & \\
   & b_3 & &     e_3  & & * \\
   & & b_4 &     e_4  & * & * \\   
   \midrule
   c_1 & \frac{c_1+c_4}{2} & c_4 &  2 & 2 & 1 \\
\end{array}
$$

Then there is no three-point subset $A\subset A'$ to which the Inductive irreducibility would be applicable. Removing either $b_1$ or $b_2$ leads to a case where $d=d'$ and $A$ is of type I2. According to Remark \ref{exaI}, the proper part of the generic symmetric curve supported on $A$ is reducible in general. Removing $b_4$ is not allowed since it would change the projection along $M^I$. Eventually, removing $b_3$ leads to a case where $d'=d/2$ and $A$ is of type I10. 
\end{exa}
\begin{exa}{\bf (The triangle)} Let $A'$ 
have the shape
$$
\begin{array}{c c c|c c c}
   b_1 & & &     e_1 \\
   & b_2 & &     e_2 \\
   & & b_3 &     e_3 \\   
   \midrule
   c_1 & c_2\ne\frac{c_1+c_3}{2} & c_3 &  2 \\
\end{array}
$$

Then there is no two-point subset $A\subset A'$ to which the Inductive irreducibility would be applicable.
\end{exa}
At least, fortunately, these examples are the only ones.
\begin{lemma}\label{lminnonexc}
A set $A'$ is non-exceptional, if and only if it contains a subset $B$ of one of the preceding four types, such that the projections of $\conv (A')$ and $\conv (B)$ along $M^I$ coincide.
\end{lemma}

\begin{figure}[h]
    \centering
    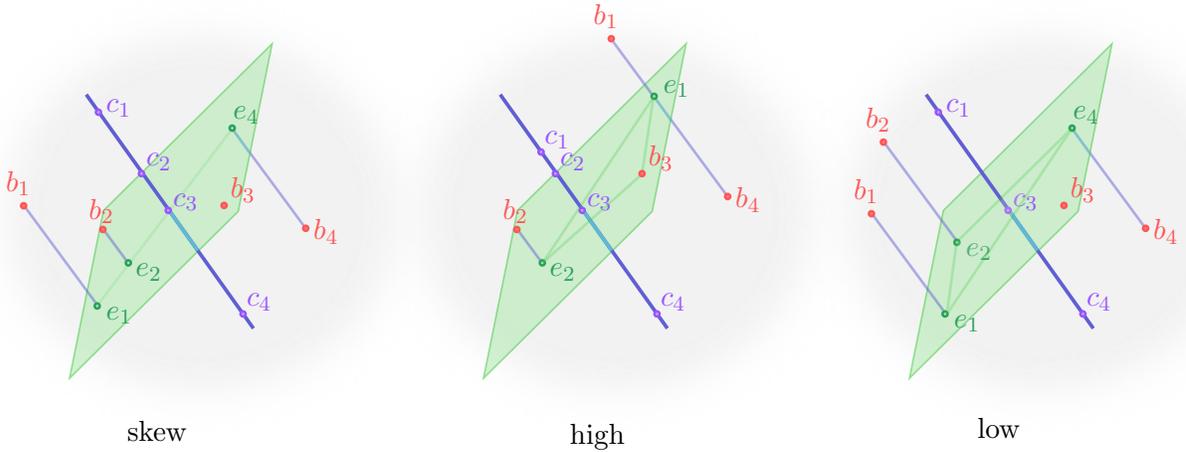
    \caption{The tetrahedra of Examples \ref{ex:skew}, \ref{ex:high} and \ref{ex:low}}
    \label{fig:minimal}
\end{figure}

{\it Proof.} It is the easy direction to check that exceptional sets $A'$ do not contain a subset $B$ of one of the preceding four types: just compare each of the five exceptional types with each of the four preceding examples.

We now discuss the hard direction: a non-exceptional $A'$ contains a subset $B$ of one of the preceding four types. Let $Q$ be the projection of $A'$ along $M^I$. Note that $Q$ has more than one point since $A'$ is not of type D.

1) Assume that each of the two end points of $Q$ has exactly one preimage $b_1, b_4\in A'$, and these two points have the same projection along $M^{-I}$. In this case, if every pair of points in $A'\setminus\{b_1,b_4\}$ were contained in a given plane containing $\{b_1,b_4\}$, then $A'$ itself would be contained in this plane, i.e. it would have the exceptional type C1. If every pair of points in $A'\setminus\{b_1,b_4\}$ were contained in a given plane parallel to $M^I$, then $A'\setminus\{b_1,b_4\}$ itself would be contained in this plane, i.e. it would have the exceptional type C2.
Thus one can find a pair of points $b_2,b_3\in A'\setminus\{b_1,b_4\}$ unlike the aforementioned pairs. Then $\{b_1,b_2,b_3,b_4\}$ is a high tetrahedron.

2) Otherwise, pick preimages $b_1, b_4\in A'$ of the two end points of $Q$ so that their projections along $M^{-I}$ are different, denote the plane through $b_1$ and $b_4$ parallel to $M^{-I}$ by $L$, and denote the set of the middle and the end points of $Q$ by $S$.

2.1) Assume that all the points of $A'$, whose projection along $M^I$ is not the middle of the segment $Q$, belong to the plane $L$. Since $A'$ is not of exceptional type I2, it contains a point $b_2$ whose projection to $A$ is not in $S$. Since $A'$ is not of exceptional type C1, it contains a point $b_3\notin L$, and this point has to project to the middle of $Q$. Thus $\{b_1,b_2,b_3,b_4\}$ is a skew tetrahedron.

2.2) Otherwise $A'$ has a point $b_2$ outside the plane $L$, whose projection to $Q$ is not the middle of $Q$.
If the projection of at least one such $b_2$ to $Q$ is not in $S$, then $\{b_1,b_2,b_4\}$ is a triangle. Otherwise there is to be $b_3\in L$ projecting to the middle of $Q$, since type I1 is forbidden. Moreover the projection of $b_3$ and $b_4$ along $M^{-I}$ are distinct since type C2 is forbidden. Thus $A'$ contains a low tetrahedron.
\hfill$\square$

\subsection{Proof of Theorem \ref{thintroirr}(1) for an $I$-nondegenerate $f\in\C^A$}
\begin{lemma} \label{lbaseindirr}
For $A'$ as in each of the four preceding examples, the proper part of the corresponding symmetric curve is irreducible.
\end{lemma}
{\it Proof.} 
If $A'$ is the triangle $\{b_1,b_2,b_3\}$ (shifted for convenience so that $b_2\in M^I$), the proper part of the symmetric curve is defined by the system $$f=g=0,\quad  g(u,v,w):=\frac{f(u,v,w)-f(v,u,w)}{u^d-v^d}.\eqno{(*)}$$ The Newton polygons of these two equations are a triangle and a quadrangle that have no parallel edges and do not belong to parallel planes. 

In order to show that \cite[Theorem 17]{kh15} is applicable to this system of equations, we should check the following two conditions.

1) No $k$ of the support sets of the equations can be shifted to the same $k$-dimensional subspace: this is obvious from the description of the support sets in our special case.

2) The system of equations is nondegenerate. 
This condition is trivially satisfied for all restricted systems except for the following two (by ``trivially'' we mean that it is satisfied with trivial genericity assumptions, i.e. for every trinomial $f$ at all):

-- the system $(*)$ itself;

-- the system $f|_E=g=0$ for the edge $E$ with the end points $b_1$ and $b_3$.

For both of them, the absence of degenerate roots follows from Theorem \ref{thsingnum}.

\vspace{1ex}

The other three cases, the low, skew and high tetrahedra, have an edge disjoint from the plane $M^I$, no matter how we shift them. Moreover, the system $(*)$ is always degenerate by Remark \ref{kh15doesnothelp}, and we cannot apply \cite[Theorem 17]{kh15} in the same way as for the triangle above.

However, if $A'$ is the low tetrahedron $\{b_1,b_2,b_3,b_4\}$ (shifted and reflected for convenience so that $b_3\in M^I$, and $b_4$ is separated from the monomial $u$ by the plane $M^I$), 
we can transform the system $(*)$ to the equivalent one $f-u^dg=g=0$, inheriting the same properties: it defines the proper part of the symmetric curve and has no degenerate roots. The Newton polygons of the latter system are a pair of triangles with no parallel edges. This again implies nondegeneracy, making \cite[Theorem 17]{kh15} applicable to the latter system (in contrast to the initial one $(*)$).

If $A'$ is a skew or high tetrahedron, we are not aware of any similar tricks reducing the question to \cite[Theorem 17]{kh15}.
However, for the skew tetrahedron, the proof of irreducibility is given in the corresponding example, and the high one is covered by Theorem \ref{thsingnum}.\hfill$\square$

\vspace{1ex}

{\it Proof of theorem \ref{thintroirr}(1) } proceeds by the cardinality of the support set $A$: the base of the induction is provided by Lemma \ref{lbaseindirr}, the induction step by Proposition \ref{propind}, and Lemma \ref{lminnonexc} ensures that the induction exhausts all non-exceptional support sets.
\hfill$\square$

\section{Discussion: what next?}\label{snext}
The problem that we study obviously extends to generic complete intersections in a torus $T\simeq\CC^n$ with an arbitrary finite group of symmetries 
$$G\subset {\rm Aut} (T)=GL(M),\quad M:={\rm Char} (T)\simeq \Z^n.$$
Besides these concordant actions of $G$ on $T$ and $M$, let us choose finite sets $A_1,\ldots,A_m\subset M$ and an action of $G$ on $\{1,\ldots,m\}$ such that $g(A_i)=A_{g(i)}$ for every $g\in G$. In the space $\C^A:=\C^{A_1}\oplus\cdots\oplus\C^{A_m}$, let us consider the vector subspace $\C^A_G$ of all tuples $f=(f_1\ldots,f_m)$, such that $f_i(g(x))=f_{g(i)}(x)$. Such a tuple defines the set $\{f=0\}\subset T$, which is symmetric under the action of $G$ on $T$ and has the same geometry for almost all $f\in \C^A_G$.
This paper is devoted to the special case when $n=3,\,m=2$, and $G\simeq\Z_2$ has the only non-trivial element $I(u,v,w)=(v,u,w)$. 

Additionally, we can lift the homomorphism $G\to S_m$ to the wreath product: $\chi:G\to S_m\wr M$. It sends every $g\in G$ to a tuple $(\sigma_g,g_1,\ldots,g_m)\in S_m\times M\times\cdots\times M$ in such a way that the equality  
$$\forall i\;:\; f_i(g(x))=g_i(x)\cdot f_{\sigma_g(i)}(x)\eqno{(*)}$$
for $g=g_2g_1$ follows from the respective equalities for $g_1$ and $g_2$. We can now define $\C^A_{G,\chi}\subset\C^A$ as the subspace of all {\it semi-invariant complete intersections} $f$ satisfying the equalities $(*)$ for every $g\in G$. Such $f$ still defines a set $\{f=0\}\subset\CC^n$ symmetric under the action of $G$. We study $\{f=0\}$ for generic $f\in\C^A_{G,\chi}$. In particular, the space $C^A_G$ equals $C^A_{G,\chi}$ for the {\it trivial} map $\chi$, such that $(g_1,\ldots,g_m)=(1,\ldots,1)$ for every $g\in G$.

For instance, if $A\subset\Z^3$ is symmetric ($A=IA$), we can study surfaces $\{f=0\}\subset\CC^3$ given by a generic equation with a symmetry $f(Ix)=f(x)$ or $f(Ix)=-f(x)$. In both cases, $G=\Z^2$, but the first case corresponds to trivial $\chi$, while the second one to $(g_1)=(-1)$ for the nontrivial element $g\in \Z_2$.

\vspace{1ex}

{\bf The general setting.} More generally (but essentially for the same price) one can study generic systems of equations with symmetries that not necessarily imply the symmetry of their solutions (such as a system $f=g=0$ in which $f(Ix)=f(x)$ and $g(Jx)=g(x)$ for two different symmetries $I$ and $J$: note that the set of solutions of this system does not have to be symmetric with respect to $I$ or $J$). Namely, choose a finite group
$$G\subset S_m\wr{\rm Aff}\,\Z^n$$
and its homomorphism $\chi:G\to S_m\wr\Z^n$.
Here ${\rm Aff}\,\Z^n$ is the group of affine automorphisms of the lattice $\Z^n$, so an element $g\in G$ defines a permutation $\sigma$ of $\{1,\ldots,m\}$ together with automorphisms $\rg_1,\ldots,\rg_m\in{\rm Aff}\,\Z^n$ and monomials $\mu_1,\ldots,\mu_m\in\Z^n$.

The collection of support sets $A_1,\ldots,A_m\subset\Z^n$ is said to be symmetric with respect to the group $G$, if, for every element $(\sigma,\rg_1,\ldots,\rg_m)$ of the group we have $$A_{\sigma(i)}=\rg_i A_i,\,i=1,\ldots,m.$$
Sending every monomial $m\in A_i$ to the monomial $\rg_i(m)\in A_{\sigma(i)}$ extends by linearity to the map $\C^{A_i}\to\C^{A_{\sigma(i)}}$ that we shall denote by $\rg_i$ as well.

For a group of symmetries $G$, its given homomorphism $\chi$, and a symmetric collection $A=(A_1,\ldots,A_m)$, we can consider the space $\C^A_{G,\chi}$ of all collections of polynomials $f_i\in\C^{A_i},\,i=1,\ldots,m$, symmetric with respect to every element $g=(\sigma,\rg_1,\ldots,\rg_m)\in G,\,\chi(g)=(\sigma,g_1,\ldots,g_m)$, in the sense that $$\rg_i(f_i)=g_i(x)\cdot f_{\sigma(i)},\,i=1,\ldots,m.$$
\begin{prb}\label{prb0}
Describe the geometry of the algebraic set $f=0$ in the torus $\CC^n$ for generic $f\in\C^A_{G,\chi}$.
\end{prb}
This paper is devoted to the special case when $n=3,\,m=2$, and $G\simeq\Z_2$ has the only non-trivial element $((1,2),I,I),\,I(u,v,w)=(v,u,w)$. We now discuss what happens as we push the (numerous) boundaries of this setting, from simple to complex.

\vspace{1ex}

{\bf Involution-symmetric spatial curves: different involutions.} Even for $n=3$ (an all the more so for higher dimensions), we have many (though finitely many) different symmetry groups $G$ isomorphic to $\Z/2\Z$. They may differ both by lattice authomorphisms (such as $((1,2),J,J)$ for $J(u,v,w)=(-u,v,w)$ or $(-u,-v,w)$) and permutations (such as $({\rm Id},I,I),\,I(u,v,w)=(v,u,w)$). The examples in the parentheses lead to other kinds of spatial symmetric curves in $\CC^3$. The first and the last ones are similar to the one that we have studied, but the second one is already quite different, because the diagonal is not a (hyper)plane, and our elementary trick with passing from $f(u,v,w)=f(v,u,w)=0$ to $(f(u,v,w)-f(v,u,w))/(u^d-v^d)=0$ won't work.

\vspace{1ex}

{\bf Involution-symmetric complete intersections in higher dimension.} As the dimension $n$ grows, even for the same number of equations $m=2$ we have more and more involutions of the lattice $\Z^n$. To deal with this diversity, it is important that most of the involutions of $\Z^n$ come from involutions of smaller lattices $\Z^k$, multiplying them by $\Z^{n-k}$ with a trivial authomorphism. 
\begin{prb}
1) To what extent one could deduce from the solution of Problem \ref{prb0} for an action of $G$ on $\Z^k$ the solution of the same problem for the trivial extension of this action on $\Z^k\oplus\Z^{n-k}$?

2) In particular, to what extend one could deduce from the results of the present paper similar irreducibility results on complete intersections $f(x_1,x_2,x_3,\ldots,x_n)=f(x_2,x_1,c_3,\ldots,x_n)=0$ for any $n$?
\end{prb}
We expect the combinatorial part of the answer to the second question to be more challenging than the geometric one (for the geometric part, see e.g. Proposition 5.1 in \cite{crit} that is already done for arbitrary $n$).

\vspace{1ex}

{\bf Involution-symmetric complete intersections in higher codimension.} As soon as $m>2$, Problem \ref{prb0} becomes drastically more complicated even for our simplest involution $I(x_1,x_2,x_3,\ldots,x_n)=(x_2,x_1,c_3,\ldots,x_n)$. E.g. even for $m=n=4$ (so that the proper part of our complete intersection $f=0$ is 0-dimensional), it is already far from trivial to count the points in the proper part, because the diagonal part has the ``wrong'' dimension 1, and because this complete intersection is not sch\"on in the sense of \cite{schon} and \cite{crit}. In case $A_i$'s belong to a coordinate hyperplane, this problem is equivalent to counting the self-intersection of a plane projection of a spatial complete intersection curve (see \cite{voorhaar}).
\begin{prb}
Generalize \cite{voorhaar} to the case when $A_i$'s do not belong to a hyperplane.
\end{prb}

The study of involution-symmetric generic complete intersections incorporates many other natural questions as special cases: for instance, the classification of collections of lattice polytopes $A_1,\ldots,A_n\subset\Z^n$ of mixed volume 1 (\cite{eg11}) is the classification of generic complete intersections of the form $$y_i-f_i(x_1,\ldots,x_n)=y_i-f_i(x'_1,\ldots,x'_n)=0,\,f_i\in\C^{A_i},\,i=1,\ldots,n,$$
which are exceptional in the sense that their proper part is empty (the involution is $I(x,x',y)=(x',x,y)$ here). This extends to higher values of the mixed volume once we proceed from involutions to more complicated symmetries.

\vspace{1ex}

{\bf Beyond involutions: Schur polynomials.} For every dimension $n$ and every number of equations $m$, we have finitely many possible finite subgroups $G$ (up to conjugation). Solving Problem \ref{prb0} for each of them even in the simplest case of a planar curve ($m=1,n=2$) is highly non-trivial. Let e.g. $G\subset S_1\wr{\rm Aff}\,\Z^2={\rm Aff}\,\Z^2$ be the group of symmetries $S_3$ of the standard hexagon $\{|u|\leqslant 1, |v|\leqslant 1,|u+v|\leqslant 1\}$, and let $\chi(\sigma)=\sgn\sigma$.

The simplest $A$ invariant under this action of $S_3$ consists of 6 points, and 
$f\in\C^A_{S_3,\sgn}$ for such $A$, up to a monomial multiplier, has the form $$\det\begin{pmatrix}
1 & 1 & 1\\
1 & x^a & x^b\\
1 & y^a & y^b
\end{pmatrix}.$$
Modulo the multiplier $(x-y)(x-1)(y-1)$, this is a Schur polynomial. Irreducibility of Schur polynomials is an important recent result in \cite{dz}. From it, similarly to the methods of Sections \ref{sconn} and \ref{sirredproof}, one can make conclusions about singularities and components of the curve $f=0$ for generic $f\in\C^A_{S_3,\sgn}$ with an arbitrary $S_3$-symmetric support set $A$, thus approaching the following problem.
\begin{prb}\label{prb4}
1) Understand the components, the singularities and the genus of an $S_3$-symmetric plane curve given by a generic polynomial $f\in\C^A_{S_3,\sgn}$ for an arbitrary $S_3$-symmetric support set $A$.

2) Do the same for $f\in\C^A_{S_3,1}$ (starting from the case of the ``Schur permanent'' for a 6-point support set $A$).

3) Extend this to symmetries of tori of higher dimension, related in a similar way to Schur polynomials and permanents of more variables.
\end{prb}
Since the other groups of symmetries of $\Z^2$ are abelian and thus much simpler than $S_3$, this sketches the approach to completely understand plane symmetric curves (in our sense) for all possible symmetries of $\Z^2$.

\vspace{1ex}

{\bf Irreducibility vs planarity.} The following conjecture may be key to solving Problem \ref{prb0}.
\begin{defin} An irreducible component of the set $f=0$ for $f\in\C^A_{G,\chi}$ is:

1) said to be planar, if it can be shifted to a proper subtorus of $\CC^n$;

2) said to be diagonal, if it can be shifted to a proper subtorus of $\CC^n$ invariant under the action of some non-unit element $g\in G$. Non-diagonal components are called proper.
\end{defin}
Consider {\it exceptional} and {\it planar} support sets:

$\mathcal{E}_{G,\chi}:=\{A \,|\,$ for a generic $f\in\C^A_{G,\chi}$ the set $f=0$ has more than one proper component$\}$;

$\mathcal{P}_{G,\chi}:=\{A \,|\,$ for a generic $f\in\C^A_{G,\chi}$ all components of $f=0$ are planar$\}$.
\begin{conj}[cf Question 1.5 in \cite{schon}]\label{conjflat}
$\mathcal{E}_{G,\chi}\subset\mathcal{P}_{G,\chi}$.
\end{conj}

The conjecture (even if confirmed only for some dimensions and symmetry groups) may greatly facilitate the classification of respective exceptional support sets, since planar support sets are much more attainable to classify.

\begin{rem} For an idea of how to approach the classification of planar support sets, see Proposition \ref{pflat} devoted to this classification in our simplest case ($n=3,G=\Z_2,\chi=\id$).

For an idea why (philosophically) planar support sets are much more attainable to classify than exceptional support sets, one may notice that irreducibility is a global property of an algebraic set, while having every irreducible component planar is a local property.\end{rem}

\vspace{1ex}

{\bf Other ground fields.} Over $\R$, already the symmetry $I(u,v)=(v,u)$ for plane curves of a given small degree is difficult and important to study, see e.g. \cite{os} and \cite{bru} for surveys and key results. 
\begin{prb}
It would be interesting to extend these classification techniques to real plane curves of small degree with the other symmetry groups ($\Z/3\Z$, $S_3$, etc.).
\end{prb}

Another natural question is to find a way to constructively (not appealing e.g. to the Lefschetz transfer principle) extend our results to arbitrary algebraically closed fields.
First of all, this would require to eliminate topological methods in our paper.

We use topology in two key points. First, it is the Morse theory in the proof of the irreducibility in Theorem \ref{thsingnum}, if we trace it back to \cite{schon}.
An alternative cohomological technique is the Dolbeault cohomology computation, as it is done for nondegenerate complete intersection in \cite{kh15}. In our case, this boils down to analyzing the Dolbeault cohomology of the toric variety with coefficients twisted by the line bundle, corresponding to the divisor of the rational function $f(x)/(x_1^d-x_2^d)$ (to which we apply the Morse theory over $\C$ in \cite{schon}). This is more combinatorially challenging, because the line bundle is not ample, but still feasible.

Another topological key step in our paper is when we deduce Proposition \ref{propind} from Proposition \ref{pflat}. Instead of this, we can notice that the reducibility of the proper part of the symmetric curve is preserved when we switch from the ground field $\C$ to the field of Puiseaux series. Over this field, in the notation of Proposition \ref{propind}, we can consider a generic symmetric curve given by the polynomial $f\in\C^{A'}$ whose coefficients have trivial valuation except for the coefficient of the monomial $a$. For the corresponding curve $f=f\circ I=0$, we can see the absence of components with the non-trivial toric span $T$, looking at its tropicalizations. It follows from a non-trivial, but purely combinatorial 
\begin{observ}\label{remtrop1}
If the intersection of tropical fans of $f=0$ and $f\circ I=0$ (for $f$ as above) contains a 1-dimensional tropical curve avoiding $0$, then this tropical curve is contained in a plane.
\end{observ}
\begin{sledst}
Proposition \ref{propind}.1 is valid over any field of characteristic 0.
\end{sledst}
We can summarize this outline as follows.
\begin{prb}
1) Prove the irreducibility in theorem \ref{thsingnum} computing the relevant Dolbeault cohomology.

2) Deduce from this the irreducibility of the proper part of an arbitrary symmetric curve by tropical methods, as in Observation \ref{remtrop1}.

3) More generally, given a generic enough symmetric complete intersection over the field of Puiseaux series, to what extent can one tell the topology and singularities of its proper part from its tropicalization?
\end{prb}

\vspace{1cm}
\noindent
A. Esterov\\
London Institute for Mathematical Sciences, UK\\
\textit{Email}: aes@lims.ac.uk\\

\noindent
L. Lang\\
Department of Electrical Engineering, Mathematics and Science, University of G\"avle, Sweden\\
\textit{Email}: lionel.lang@hig.se

\end{document}